\documentclass[11pt]{amsart}   	
\usepackage[margin=1in]{geometry}
\usepackage{subfigure,verbatim}									
\usepackage{amsmath,amssymb,amsfonts,amsthm}
\usepackage{hyperref}
\usepackage{mathtools}
\usepackage{tikz}
\usetikzlibrary{cd}
\usetikzlibrary{decorations.markings}
\usepackage[shortlabels]{enumitem}
\usepackage{makecell}
\setcellgapes{4pt}
\usepackage{mathdots} 

\usepackage[noabbrev,capitalise,nameinlink]{cleveref}
\usepackage{multicol}

\hypersetup{colorlinks=true, citecolor=lavender, linkcolor=lavender,urlcolor=lavender}

\definecolor{lavender}{rgb}{0.5,0,1.0}
\definecolor{NormalGreen}{RGB}{0,220,0}
\definecolor{ChillBlue}{RGB}{0,182,255}
\definecolor{MyOrange}{RGB}{255,150,0}
\definecolor{Maroon}{RGB}{150,0,0}
\definecolor{Pink}{RGB}{255,0,255}
\definecolor{MyPurple}{RGB}{180,0,255} 

\newtheorem{theorem}{Theorem}[section]
\newtheorem*{theorem*}{Theorem}
\newtheorem{lemma}[theorem]{Lemma}
\newtheorem{corollary}[theorem]{Corollary}
\newtheorem{proposition}[theorem]{Proposition}

\newtheorem*{thma}{\Cref{thm:near-homomesy}}
\newtheorem*{thmb}{\Cref{thm:OrderComparedToDeflation}}
\newtheorem*{thmc}{\Cref{thm:swap_orthomesy}}
\newtheorem*{thmd}{\Cref{cor:swap-deflation-orthomesy}}
\newtheorem*{thme}{\Cref{Z4_totalsum}}

\theoremstyle{definition}
\newtheorem{definition}[theorem]{Definition}

\newtheorem*{define*}{Definition}
\newtheorem{example}[theorem]{Example}

\newtheorem*{notation*}{Notation}

\newtheorem{remark}[theorem]{Remark}
\newtheorem*{remark*}{Remark}

\newcommand{\as}{\mathcal{A}}
\newcommand{\ase}{\mathcal{A}_e}
\newcommand{\asi}{\mathcal{A}_i}
\newcommand{\tot}{\mathrm{Tot}}
\newcommand{\rev}{\mathrm{rev}}

\DeclareMathOperator{\boxempty}{\boxed{ }}

\DeclareMathOperator{\Inc}{\mathrm{Inc}}
\DeclareMathOperator{\Pro}{\mathrm{Pro}}
\DeclareMathOperator{\jdt}{\mathrm{jdt}}
\DeclareMathOperator{\Con}{\mathrm{Con}}

\newcommand{\swap}{\textnormal{swap}}
\newcommand{\lcm}{\text{lcm}}
\newcommand{\calZ}{\mathcal{Z}}
\newcommand{\xdownarrow}[1]{%
  {\left\downarrow\vbox to #1{}\right.\kern-\nulldelimiterspace}
}

\usepackage{soul}

\usepackage{mathdots}

\title{Orthomesy and promotion on self-dual posets}
\author{Esther Banaian}
\address{
Department of Mathematics, University of California, Riverside, Riverside, CA 92507 USA}\email{estherbanaian@gmail.com}
\author{Emily Barnard}
\address{Department of Mathematical Sciences, DePaul University, Chicago, IL 60604 USA}\email{e.barnard@depaul.edu}

\author{Sunita Chepuri}
\address{Department of Mathematics and Computer Science, University of Puget Sound, Tacoma, WA 98406 USA}
\email{schepuri@pugetsound.edu}

\author{Jessica Striker}
\address{Department of Mathematics, North Dakota State University, Fargo, ND 58108 USA}
\email{jessica.striker@ndsu.edu}

\begin{document}

\begin{abstract}
We introduce the notion of orthomesy, which is related to  homomesy, a central phenomenon in dynamical algebraic combinatorics.
An orbit is said to be orthomesic with respect to a statistic if the orbit's average statistic value is equal to the global average.
We particularly focus on the action of promotion on increasing labelings of certain fence posets called zig-zag posets, and two statistics, the antipodal sum statistic and the total sum statistic.
We classify all of the orthomesic promotion orbits for the zig-zag poset with four elements.
Along the way, we investigate how homomesy of one action can be used to find orthomesic orbits for another action, for the same fixed statistic.
We prove several general results which can be used to find infinite families of orthomesic orbits for any self-dual poset that has an order-reversing involution.
\end{abstract}

\maketitle


\section{Introduction}
\label{sec:intro}
Let $X$ be a finite set, let $\phi:X\to X$ be a bijective action on $X$, and let $\eta:X\rightarrow \mathbb{Z}$ be a statistic on $X$.
\emph{Homomesy} is a central phenomenon in dynamical algebraic combinatorics in which the average statistic value of each orbit of $\phi$ equals the global average
(equivalently, each orbit has the same average statistic value).  
Defined by  Propp and Roby in  \cite{PR2015}, where they studied its occurrence with respect to \emph{rowmotion} on order ideals of  posets, homomesy has been found in many other contexts (see e.g.~the survey articles \cite{Roby2016,Striker2017}). Recently, the study of orbit averages has expanded to instances in which homomesy does not occur, but there is still a notable pattern to the orbit averages. In \cite{elizalde2021rowmotion}, the authors define \emph{homometry}, in which the average statistic value on an orbit is the same as the average value on all other orbits of the same size. 
The authors of \cite{elizalde2021rowmotion} study rowmotion of order ideals of \emph{fence posets} with respect to the statistic $\hat{\chi}$ which measures the size of an order ideal.
A fence poset is a poset of the form $x_1 \lessdot x_2 \lessdot \cdots \lessdot x_a \gtrdot x_{a+1} \gtrdot \cdots \gtrdot x_{a+b} \lessdot x_{a+b+1} \cdots$, as pictured below.
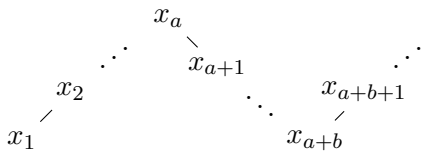
\begin{center}
\begin{tikzpicture}[scale=0.65]
    \node (x1) at (0,0) {$x_1$};
    \node (x2) at (1,1) {$x_2$};
    \node(dots1) at (1.875,1.875){$\iddots$};
    \node (xa) at (3,2.5){$x_a$};
    \node(xa+) at (4,1.5){$x_{a+1}$};
    \draw(x1) -- (x2);
    \draw(xa) -- (xa+);
    \node(dots2) at (4.875,1.5 - .625){$\ddots$};
    \node(xab) at (6,0){$x_{a+b}$};
    \node(xab+) at (7,1){$x_{a+b+1}$};
    \draw(xab) -- (xab+);
    \node(dots3) at (7.875,1.875){$\iddots$};
\end{tikzpicture}
\end{center}
Fence posets are prominent in several areas of mathematics, including the study of gentle and string algebras and cluster algebras \cite{ccanakcci2018cluster}. Fence posets have been of particular interest lately due to a conjecture of Morier-Genoud and Ovsienko \cite{morier2020continued}, later proven by O{\u{g}}uz and Ravichandran \cite{ouguz2023rank}, which stated that the rank generating function of the lattice of order ideals of a fence poset is unimodal.
Importantly, rowmotion of order ideals of fence posets is \emph{not} generally homomesic.
One of the main results of \cite{elizalde2021rowmotion} shows  that for certain fence posets, rowmotion of order ideals is homometric with respect to $\hat{\chi}$ (see \cite[Theorem~4.5]{elizalde2021rowmotion}).

In this paper, we continue the study of fence posets initiated by \cite{elizalde2021rowmotion}, this time focusing on the action of \emph{promotion} on the set of increasing labelings of a family of fence posets called \emph{zig-zag posets}, which we denote $\mathcal{Z}_n$, with respect to two statistics, the antipodal sum and total sum statistics.
More precisely, $\mathcal{Z}_n$ is the fence poset with $n$ elements satisfying: $x_1\lessdot x_2 \gtrdot x_3\lessdot x_4\gtrdot \cdots$.
For any finite poset, we define an increasing labeling of $P$ to be a function $f: P \to \mathbb{N}$ such that whenever $x < y$ in $P$, we have $f(x) < f(y)$. Denote  the set of all increasing labelings of $P$ with values in $[q]:=\{1,2,\ldots,q\}$ by  $\Inc^q(P)$.
See \Cref{def:IncrLabeling}.
The total sum statistic of $f\in \Inc^q(P)$ simply sums the labels $f(x)$ for all $x\in P$.
If $P$ is self-dual (see Definition \ref{def:SelfDual}) and has an order-reversing involution $\kappa:P\to P$,  the \emph{antipodal sum statistic with respect to $x$} for a fixed $x \in P$ is the sum $f(x)+f(\kappa(x))$, and is denoted  $\mathcal{A}_x(f)$.
See \Cref{def:stats}.

The promotion action on $\Inc^q(P)$ was defined in \cite{dilks2019rowmotion} as a generalization of \emph{K-promotion} \cite{Pechenik2014} and \emph{K-jeu de taquin} \cite{TY2009} on \emph{increasing tableaux}. These are themselves generalizations of promotion and jeu de taquin on standard and semistandard Young tableaux \cite{Sch1972}. Increasing labeling promotion was shown in \cite{dilks2019rowmotion} to be equivariant with \emph{rowmotion} of order ideals of a related poset they called the \emph{gamma poset}; this is a natural analogue of results on increasing tableaux \cite{dilks2017resonance} and two-row standard Young tableaux \cite{SW2012}.

Homomesy has been studied with respect to both promotion and rowmotion, and occurs in the following instances, among many others: cardinality under rowmotion on order ideals of the product of two chains poset \cite{PR2015}, other minuscule posets \cite{RushWang}, and the product of three chains where one is of size $2$ \cite{Vorland};  sums of centrally symmetric entries under promotion of rectangular semistandard Young tableaux \cite{BPS}; and sums of centrally symmetric entries of the first/last row/column under $K$-promotion of rectangular increasing  tableaux \cite{Pechenik2014,Pechenik2017}.  
The interplay between promotion and rowmotion makes the study of promotion on fence posets a natural next step from \cite{elizalde2021rowmotion}.

Just as rowmotion of order ideals is generally not homomesic for fence posets, promotion also generally fails to be homomesic (with respect to the antipodal sum and total sum statistics).
Rather than focusing on homometry, we introduce a different weakening of homomesy called \emph{orthomesy}.

\begin{definition}
\label{def:orthomesy}
Let $\mathcal{O}$ be an orbit of a bijective action on a finite set $X$ and $\eta:X\rightarrow \mathbb{Z}$ be a statistic on $X$. Then the orbit $\mathcal{O}$ exhibits \emph{orthomesy} (or is \emph{orthomesic}) with respect to the statistic $\eta$ if 
\[\dfrac{1}{|\mathcal{O}|}\displaystyle\sum_{x\in \mathcal{O}}\eta(x)=\dfrac{1}{|X|}\sum_{x\in X}\eta(x).\]  
\end{definition}
Note that an orbit is trivially orthomesic if $\mathcal{O}=X$, that is, if the action has only one orbit.
If all orbits are orthomesic, then the orthomesy is a homomesy. 
So orthomesy allows one to move beyond the binary question of whether or not a set, statistic, and action exhibits homomesy to the aim of characterizing which orbits exhibit orthomesy.
In our main result, we classify all of the orthomesic orbits of the promotion action on increasing labelings of the zig-zag poset $\mathcal{Z}_4$ with respect to the antipodal sum statistic $\mathcal{A}_x$, where $x$ is any element of $\mathcal{Z}_4$.

\begin{thma}
    Let $P=\calZ_4$ and $x\in P$. Let $\mathcal{O}$ be a promotion orbit of $\Inc^q(P)$. Then $\mathcal{O}$ exhibits orthomesy with respect to the antipodal sum statistic $\as_x$ if and only if $\mathcal{O}$ avoids $1324$ or is balanced.
\end{thma}
\noindent
Both the pattern avoidance condition and the balanced condition are easy to check (see \Cref{def:pattern,def:Balanced}).

In order to prove this result, we develop two general tools which can be applied to any finite poset $P$.  The first allows us to understand the infinite set of all increasing labelings $\bigcup_{q\in\mathbb{Z}} \Inc^q(P)$ through a finite set of \emph{packed} increasing labelings.  An increasing labeling is packed if it has no ``holes."  The key point is if $f\in \bigcup_{q\in\mathbb{Z_{>0}}} \Inc^q(P)$ is packed then, $q$ is at most equal to the cardinality of $P$.  We can relate any increasing labeling to a packed labeling through a process known as \emph{deflation}. An increasing labeling can be reconstructed from its deflation and \emph{binary content word}, $\Con(f)$. See~\cref{def:Con} for more details.
Mandel and Pechenik in \cite[Theorem 6.1]{mandel2018orbits} computed the order of the promotion action on any increasing labeling in terms of packed labelings and binary content words for minuscule posets.
The next theorem extends \cite[Theorem 6.1]{mandel2018orbits} to any finite poset.
\begin{thmb}
Let $P$ be any finite poset.
Let $f \in \Inc^q(P)$ such that $\Con(f)$ has period $\ell$. Let~$r$ denote the smallest integer such that the deflation $\overline{f}$ is in $\Inc^r(P)$, and suppose the size of the promotion orbit containing $\overline{f}$ is $\tau$. Then, the promotion orbit of $f$ has size \[
\frac{\tau \ell}{\gcd(r\ell/q,\tau)}.
\]
\end{thmb}

The above theorem tells us that the order of promotion on the set $\Inc^q(P)$ is divisible by $q$, as long as $q>|P|$ (see \Cref{cor:Divisible}).
For small $n$, we use \Cref{thm:OrderComparedToDeflation} to explicitly compute the order of promotion for zig-zag posets (see \Cref{thm:linear-Z3,thm:linear-Z4,thm:linear-Z5}).
We also use the ideas of packed labelings and deflation in \Cref{linear_ext,prop:Stable}, where we provide conditions under which orthomesy of a deflated promotion orbit $\overline{\mathcal{O}}$ can be lifted to orthomesy of $\mathcal{O}$ for the total sum statistic and the antipodal sum statistic, respectively.

Our second general tool is to analyze the relationship between promotion and an involution which we call $\emph{swap}$.
The swap map is inspired by order ideal complementation and its role in the study of rowmotion on fence posets.
As in \cite{elizalde2021rowmotion}, the swap action requires our poset $P$ to be self-dual with an order-reversing involution. 
The key observation is that the swap action is clearly homomesic with respect to the total sum statistic and the antipodal sum statistic $\mathcal{A}_x$, for any $x$ in $P$.
We use this fact to show the following.

\begin{thmc}
Let $P$ be a self-dual poset with at least one order-reversing involution and let $\mathcal{O}$ be a promotion orbit of elements of $\Inc^q(P)$ such that $\swap(\mathcal{O}) = \mathcal{O}$. Then, $\mathcal{O}$ exhibits orthomesy with respect to the antipodal sum statistic $\mathcal{A}_x$ for all $x \in P$ and with respect to the total sum statistic $\tot$.
\end{thmc}

Furthermore, we show that swap behaves well with deflation and ``anticommutes'' with promotion (see \Cref{lem:SwapAndDeflation} and \Cref{thm:SwapAndPro}).
This leads to the following result, which allows us to  lift the orthomesy of a deflated orbit $\overline{\mathcal{O}}$ up to $\mathcal{O}$.

\begin{thmd}
    Let $P$ be a self-dual poset with at least one order-reversing involution.  Let $f\in\Inc^q(P)$ and let $\mathcal{O}$ be the promotion orbit of $f$ in $\Inc^q(P)$.  Let $r$ be the largest label in $\overline{f}$ and $\overline{\mathcal{O}}$ be the orbit of $\overline{f}\in\Inc^r(P)$.  Suppose the following hold:
    \begin{itemize}
        \item $\rho^i(\Con(f))=\rev(\Con(f))$ for some $i\in[q]$
        \item $|\mathcal{O}|=|\overline{\mathcal{O}}| \ell$ where $\ell$ is the period of $\Con(f)$.
        \item $\overline{\mathcal{O}}$ is swap-closed.
    \end{itemize}
    Then the orbit $\mathcal{O}$ is orthomesic with respect to the antipodal sum statistic $\as_x$ for all $x\in P$ and the total sum statistic $\tot$.

\end{thmd}

We combine all of these tools to classify orthomesic promotion orbits of $\Inc^q(\mathcal{Z}_4)$, with respect to the antipodal sum statistic.
Surprisingly, even though the antipodal sum statistic is not homomesic, the total sum statistic is homomesic.

\begin{thme}
Promotion on $\Inc^q(\calZ_4)$ is $2(q+1)$-mesic with respect to the total sum statistic.
\end{thme}

This paper is organized as follows.
In \Cref{sec:Orthomesy}, we provide examples of orthomesy and describe
a relation between orthomesy and homomesy,
which is a component of the proof of \Cref{thm:swap_orthomesy}. 
In \Cref{sec:PromotionBasics}, we define increasing labelings, promotion, and deflation and prove~\Cref{thm:OrderComparedToDeflation}.
In \Cref{sec:lin_growth} we explicitly compute the order of promotion on increasing labelings of $\mathcal{Z}_n$ when $n=3,4,$ and $5$.
In \Cref{sec:swap} we analyze the relationship between swap and promotion.
In \Cref{sec:deflation}, we describe when orthomesy of a deflated orbit $\overline{\mathcal{O}}$ can be lifted to $\mathcal{O}$.
In \Cref{sec:Z4}, we prove \Cref{thm:near-homomesy}, classifying all orthomesic orbits of $\mathcal{Z}_4$ with respect to the antipodal sum statistic, and \Cref{Z4_totalsum}, giving a total sum homomesy.
We conclude with a section on future directions.


\section{Orthomesy}\label{sec:Orthomesy}
In this short section, we give two initial examples of orthomesy (see \Cref{def:orthomesy}) and prove \Cref{prop:homomesy-to-orthomesy} relating homomesic and orthomesic actions. Then starting in the next section through the end of the paper, we study promotion on increasing labelings, which is our main application of orthomesy and the context in which we isolated this phenomenon.

\begin{example} Consider the reverse action $\mathcal{R}:S_n\rightarrow S_n$ which acts on a permutation $\pi$ as $\mathcal{R}(\pi_1\pi_2\ldots\pi_n)=\pi_n\ldots\pi_2\pi_1$; for example, $\mathcal{R}(316524)=425613$. 
Consider the statistic given by the first entry of the permutation. The global average of this statistic is $\frac{n+1}{2}$, since this is the average of the numbers $1$ through $n$ and each number is equally likely to be in the first entry. Thus the $\mathcal{R}$-orbit of $\pi\in S_n$ is orthomesic with respect to the first entry statistic if and only if the first and last entries of $\pi$ sum to $n+1$, as in the case of the example orbit.
\end{example}

We now move to an example of orthomesy involving rowmotion, which is a fundamental action in dynamical algebraic combinatorics (see e.g.~\cite{SW2012,Striker2017} for further information). 

\begin{definition}\label{def:SelfDual}
    A finite poset $P$ is self-dual if there exists a bijection $\iota:P\rightarrow P$ such that $x<y$ implies $\iota(y)<\iota(x)$.
\end{definition}
Note that this definition does not imply that $\iota$ is an involution; for a counterexample, see \cite[Fig.~3.69]{EC1}. We will need an order-reversing involution in several of our results later in this paper; such an involution exists naturally in many posets of interest. 

\begin{example}
Let $P$ be a self-dual poset with order-reversing bijection $\iota:P\rightarrow P$. An \emph{order ideal} of $P$ is a subset $I$ of $P$ such that if $y\in I$ and $x<y$ then $x\in I$; the set of order ideals of $P$ is denoted $J(P)$. The \emph{dual order ideal} of $I$ (called the \emph{ideal complement} in \cite[p.4]{elizalde2021rowmotion}) is the image of its complement $P\setminus I$ under $\iota$. 
The \emph{rowmotion} of $I\in J(P)$ is the order ideal generated by the minimal elements of $P\setminus I$.  We call a rowmotion orbit $\mathcal{O}$  \emph{self-dual} if for all $I\in \mathcal{O}$, the dual order ideal of $I$ is also in $\mathcal{O}$. The following is a restatement of \cite[Corollary 3.3(a)]{elizalde2021rowmotion} in the language of orthomesy:
    Given a self-dual poset, 
    any self-dual  rowmotion orbit exhibits orthomesy  with respect to the cardinality statistic. 
\end{example}

Note that in the above, the map from an order ideal to its dual order ideal is clearly homomesic with respect to the cardinality statistic.
This is an example of the following general phenomenon which we describe explicitly in the next proposition. For a fixed statistic (such as cardinality), if one action, $\phi$, (such as order ideal dual) is homomesic, this can give rise to orthomesies for another action, $\psi$ (such as rowmotion).
One must simply check that a $\psi$ orbit is ``covered'' by $\phi$-orbits.  More precisely, given $\phi,\psi:X\rightarrow X$, an orbit $\mathcal{O}$ of $\psi$ is \emph{$\phi$-closed} if for all $x\in \mathcal{O}$, $\phi(x)\in\mathcal{O}$.

\begin{proposition}\label{prop:homomesy-to-orthomesy}
Let $X$ be a finite set and $\eta:X\to \mathbb{Z}$ be a statistic on $X$.
Suppose $\phi$ and $\psi$ are bijective actions on $X$.  If $(X,\eta, \phi)$ exhibits homomesy, then any orbit of $\psi$ which is $\phi$-closed exhibits orthomesy with respect to $\eta$.
\end{proposition}
\begin{proof}
Suppose that the global average of the $\eta$-statistic on $X$ is equal to $n$.
Then for every $\phi$-orbit $\mathcal{O}$ we have $\sum_{x\in\mathcal{O}} \eta(x) = n |\mathcal{O}|$.
Suppose that $\mathcal{O}_{\psi}$ is an orbit of the $\psi$ action which is closed under the action of $\phi.$
This implies that $\mathcal{O}_{\psi}$ can be partitioned by the orbits of $\phi$.
That is, $\mathcal{O}_{\psi}=\bigsqcup_{i=1}^k\mathcal{O}_{\phi, i}$ where each $\mathcal{O}_{\phi, i}$ is an orbit of $\phi$.  Then we have 
\[\sum_{x\in \mathcal{O}_{\psi}} \eta(x) = \sum_{i=1}^k \left(\sum_{x\in \mathcal{O}_{\phi,i}} \eta(x)\right) = \sum_{i=1}^k n |\mathcal{O}_{\phi,i}| = n |\mathcal{O}_{\psi}|.\]
Thus the average value of $\eta$ on $\mathcal{O}_{\psi}$ is equal to the global average.
\end{proof}

We later use this proposition to prove an orthomesy involving promotion on increasing labelings of self-dual posets (see \Cref{thm:swap_orthomesy}).

\section{Increasing labeling promotion and deflation}\label{sec:PromotionBasics}

We developed the idea of orthomesy when investigating the action of increasing labeling promotion.  The rest of the paper is dedicated to exploring this action  and several of its orthomesies.

Let $P$ be a finite poset and $\mathbb{N}$ denote the set of positive integers.
\begin{definition}\label{def:IncrLabeling}
We say $f: P \to \mathbb{N}$ is an \emph{increasing labeling} of $P$ if whenever $x, y \in P$, $x < y$, we have $f(x) < f(y)$. Denote  the set of all increasing labelings of $P$ with values in $[q]$ by  $\Inc^q(P)$.
\end{definition}

\begin{example}\label{ex:FirstExOfIncrLabeling}
Below on the left we draw a poset $P$ and on the right we depict an increasing labeling, $f$, with values in $[9]$ by writing $f(x)$ at each element $x \in P$.  

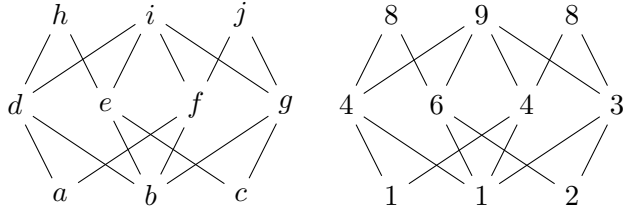
\begin{center}
\begin{tabular}{cc}
\begin{tikzpicture}[scale=1.2]
\node(a) at (0.5,0){$a$};
\node(b) at (1.5,0){$b$};
\node(c) at (2.5,0){$c$};
\node(d) at (0,1){$d$};
\node(e) at (1,1){$e$};
\node(f) at (2,1){$f$};
\node(g) at (3,1){$g$};
\node(h) at (0.5,2){$h$};
\node(i) at (1.5,2){$i$};
\node(j) at (2.5,2){$j$};
\draw(a) -- (d);
\draw(a) -- (f);
\draw(b) -- (d);
\draw(b) -- (e);
\draw(b) -- (f);
\draw(b) -- (g);
\draw(c) -- (e);
\draw(c) -- (g);
\draw(d) -- (h);
\draw(e) -- (h);
\draw(f) -- (j);
\draw(g) -- (j);
\draw(d) -- (i);
\draw(e) -- (i);
\draw(f) -- (i);
\draw(g) -- (i);
\end{tikzpicture}   
&
\begin{tikzpicture}[scale=1.2]
\node(a) at (0.5,0){1};
\node(b) at (1.5,0){1};
\node(c) at (2.5,0){2};
\node(d) at (0,1){4};
\node(e) at (1,1){6};
\node(f) at (2,1){4};
\node(g) at (3,1){3};
\node(h) at (0.5,2){8};
\node(i) at (1.5,2){9};
\node(j) at (2.5,2){8};
\draw(a) -- (d);
\draw(a) -- (f);
\draw(b) -- (d);
\draw(b) -- (e);
\draw(b) -- (f);
\draw(b) -- (g);
\draw(c) -- (e);
\draw(c) -- (g);
\draw(d) -- (h);
\draw(e) -- (h);
\draw(f) -- (j);
\draw(g) -- (j);
\draw(d) -- (i);
\draw(e) -- (i);
\draw(f) -- (i);
\draw(g) -- (i);
\end{tikzpicture}  \\
\end{tabular}
\end{center}
\end{example}

Increasing labelings are a generalization of increasing tableaux, which are themselves generalizations of standard Young tableaux. Promotion is a well-studied action on tableaux. Following \cite{dilks2019rowmotion}, we study promotion on increasing labelings.

\begin{definition}[\protect{\cite[Def.\ 3.3]{dilks2019rowmotion}}]
\label{def:ProJDT}
Let $\mathbb{N}_{\boxempty}(P)$ denote the set of labelings $g:P\rightarrow(\mathbb{N}\cup\boxempty)$.
Define the $i$th \emph{jeu de taquin slide} $\sigma_i:\mathbb{N}_{\boxempty}(P)\rightarrow \mathbb{N}_{\boxempty}(P)$ 
as follows: 

\[\sigma_i(g)(x)=\begin{cases}
i &  g(x)=\boxempty \mbox{ and } g(y)=i \mbox{ for some } y\gtrdot x \\
\boxempty &  g(x)=i \mbox{ and } g(z)=\boxempty \mbox{ for some } z\lessdot x\\
g(x) & \mbox{otherwise.}
\end{cases}\]
That is, $\sigma_i(g)(x)$ replaces a label $\boxempty$ with $i$ if $i$ is the label of a cover of $x$, replaces a label $i$ by $\boxempty$ if $x$ covers an element labeled by $\boxempty$, and leaves all other labels unchanged.

Let $\sigma_{i\rightarrow j}:\mathbb{N}_{\boxempty}(P)\rightarrow \mathbb{N}_{\boxempty}(P)$ be defined as \[\sigma_{i\rightarrow j}(g)(x)=\begin{cases} j & g(x)=i\\  g(x) &\mbox{otherwise}.\end{cases}\]
That is, $\sigma_{i\rightarrow j}(g)(x)$ replaces all labels $i$ by $j$. 

For $f\in\Inc^q(P)$, let $\jdt(f)=\sigma_{\boxempty\rightarrow (q+1)}\circ\sigma_{q}\circ\sigma_{q-1}\circ\cdots\circ\sigma_{3}\circ\sigma_2\circ\sigma_{1\rightarrow \boxempty}(f)$. That is, first replace all $1$ labels by $\boxempty$. Then perform the $i$th jeu de taquin slide $\sigma_i$ for all $2\leq i\leq q$. Next, replace all labels $\boxempty$ by $q+1$. 
Define \emph{jeu de taquin promotion} on $f$ as $\Pro(f)(x)=\jdt(f)(x)-1$.
\end{definition}

\begin{example}\label{ex:Pro}
Here we demonstrate \Cref{def:ProJDT} by computing the increasing labeling $\Pro(f)$ for the increasing labeling $f$ from \Cref{ex:FirstExOfIncrLabeling}.

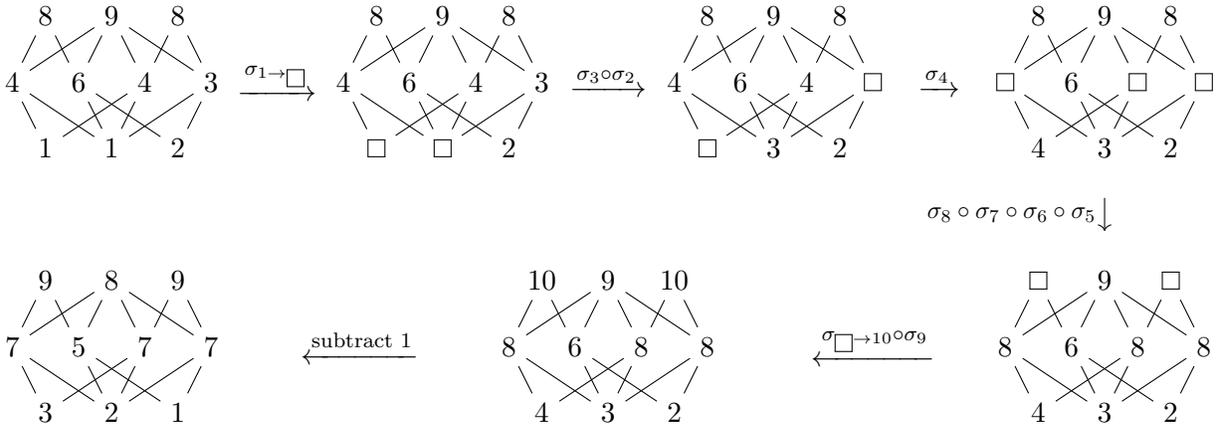
\begin{center}
\begin{tikzpicture}[scale=0.88]
\node(a1) at (0.5,0){1};
\node(b1) at (1.5,0){1};
\node(c1) at (2.5,0){2};
\node(d1) at (0,1){4};
\node(e1) at (1,1){6};
\node(f1) at (2,1){4};
\node(g1) at (3,1){3};
\node(h1) at (0.5,2){8};
\node(i1) at (1.5,2){9};
\node(j1) at (2.5,2){8};
\draw(a1) -- (d1);
\draw(a1) -- (f1);
\draw(b1) -- (d1);
\draw(b1) -- (e1);
\draw(b1) -- (f1);
\draw(b1) -- (g1);
\draw(c1) -- (e1);
\draw(c1) -- (g1);
\draw(d1) -- (h1);
\draw(e1) -- (h1);
\draw(f1) -- (j1);
\draw(g1) -- (j1);
\draw(d1) -- (i1);
\draw(e1) -- (i1);
\draw(f1) -- (i1);
\draw(g1) -- (i1);
\node[] at (4,1){$\xrightarrow[]{\sigma_{1\rightarrow \boxempty}}$};
\node(a2) at (5.5,0){$\boxempty$};
\node(b2) at (6.5,0){$\boxempty$};
\node(c2) at (7.5,0){2};
\node(d2) at (5,1){4};
\node(e2) at (6,1){6};
\node(f2) at (7,1){4};
\node(g2) at (8,1){3};
\node(h2) at (5.5,2){8};
\node(i2) at (6.5,2){9};
\node(j2) at (7.5,2){8};
\draw(a2) -- (d2);
\draw(a2) -- (f2);
\draw(b2) -- (d2);
\draw(b2) -- (e2);
\draw(b2) -- (f2);
\draw(b2) -- (g2);
\draw(c2) -- (e2);
\draw(c2) -- (g2);
\draw(d2) -- (h2);
\draw(e2) -- (h2);
\draw(f2) -- (j2);
\draw(g2) -- (j2);
\draw(d2) -- (i2);
\draw(e2) -- (i2);
\draw(f2) -- (i2);
\draw(g2) -- (i2);
\node[] at (9,1){$\xrightarrow[]{\sigma_{3}\circ \sigma_2}$};
\node(a3) at (10.5,0){$\boxempty$};
\node(b3) at (11.5,0){3};
\node(c3) at (12.5,0){2};
\node(d3) at (10,1){4};
\node(e3) at (11,1){6};
\node(f3) at (12,1){4};
\node(g3) at (13,1){$\boxempty$};
\node(h3) at (10.5,2){8};
\node(i3) at (11.5,2){9};
\node(j3) at (12.5,2){8};
\draw(a3) -- (d3);
\draw(a3) -- (f3);
\draw(b3) -- (d3);
\draw(b3) -- (e3);
\draw(b3) -- (f3);
\draw(b3) -- (g3);
\draw(c3) -- (e3);
\draw(c3) -- (g3);
\draw(d3) -- (h3);
\draw(e3) -- (h3);
\draw(f3) -- (j3);
\draw(g3) -- (j3);
\draw(d3) -- (i3);
\draw(e3) -- (i3);
\draw(f3) -- (i3);
\draw(g3) -- (i3);
\node[] at (14,1){$\xrightarrow[]{\sigma_{4}}$};
\node(a4) at (15.5,0){$4$};
\node(b4) at (16.5,0){$3$};
\node(c4) at (17.5,0){2};
\node(d4) at (15,1){$\boxempty$};
\node(e4) at (16,1){6};
\node(f4) at (17,1){$\boxempty$};
\node(g4) at (18,1){$\boxempty$};
\node(h4) at (15.5,2){8};
\node(i4) at (16.5,2){9};
\node(j4) at (17.5,2){8};
\draw(a4) -- (d4);
\draw(a4) -- (f4);
\draw(b4) -- (d4);
\draw(b4) -- (e4);
\draw(b4) -- (f4);
\draw(b4) -- (g4);
\draw(c4) -- (e4);
\draw(c4) -- (g4);
\draw(d4) -- (h4);
\draw(e4) -- (h4);
\draw(f4) -- (j4);
\draw(g4) -- (j4);
\draw(d4) -- (i4);
\draw(e4) -- (i4);
\draw(f4) -- (i4);
\draw(g4) -- (i4);
\draw[->](16.5,-0.75) -- (16.5,-1.25);
\node[left,scale=0.85] at (16.5,-1){$\sigma_8 \circ \sigma_7 \circ \sigma_6 \circ \sigma_5$};
\node(a5) at (0.5,-4){3};
\node(b5) at (1.5,-4){2};
\node(c5) at (2.5,-4){1};
\node(d5) at (0,-3){7};
\node(e5) at (1,-3){5};
\node(f5) at (2,-3){7};
\node(g5) at (3,-3){7};
\node(h5) at (0.5,-2){9};
\node(i5) at (1.5,-2){8};
\node(j5) at (2.5,-2){9};
\draw(a5) -- (d5);
\draw(a5) -- (f5);
\draw(b5) -- (d5);
\draw(b5) -- (e5);
\draw(b5) -- (f5);
\draw(b5) -- (g5);
\draw(c5) -- (e5);
\draw(c5) -- (g5);
\draw(d5) -- (h5);
\draw(e5) -- (h5);
\draw(f5) -- (j5);
\draw(g5) -- (j5);
\draw(d5) -- (i5);
\draw(e5) -- (i5);
\draw(f5) -- (i5);
\draw(g5) -- (i5);
\node[] at (3+2.25,-3){$\xleftarrow[]{\text{subtract 1}}$};
\node(a7) at (8,-4){4};
\node(b7) at (9,-4){3};
\node(c7) at (10,-4){2};
\node(d7) at (7.5,-3){8};
\node(e7) at (8.5,-3){6};
\node(f7) at (9.5,-3){8};
\node(g7) at (10.5,-3){8};
\node(h7) at (8,-2){10};
\node(i7) at (9,-2){9};
\node(j7) at (10,-2){10};
\draw(a7) -- (d7);
\draw(a7) -- (f7);
\draw(b7) -- (d7);
\draw(b7) -- (e7);
\draw(b7) -- (f7);
\draw(b7) -- (g7);
\draw(c7) -- (e7);
\draw(c7) -- (g7);
\draw(d7) -- (h7);
\draw(e7) -- (h7);
\draw(f7) -- (j7);
\draw(g7) -- (j7);
\draw(d7) -- (i7);
\draw(e7) -- (i7);
\draw(f7) -- (i7);
\draw(g7) -- (i7);
\node[] at (13,-3){$\xleftarrow[]{\sigma_{\boxempty \rightarrow 10}\circ \sigma_9}$};
\node(a8) at (15.5,-4){$4$};
\node(b8) at (16.5,-4){$3$};
\node(c8) at (17.5,-4){2};
\node(d8) at (15,-3){8};
\node(e8) at (16,-3){6};
\node(f8) at (17,-3){8};
\node(g8) at (18,-3){8};
\node(h8) at (15.5,-2){$\boxempty$};
\node(i8) at (16.5,-2){9};
\node(j8) at (17.5,-2){$\boxempty$};
\draw(a8) -- (d8);
\draw(a8) -- (f8);
\draw(b8) -- (d8);
\draw(b8) -- (e8);
\draw(b8) -- (f8);
\draw(b8) -- (g8);
\draw(c8) -- (e8);
\draw(c8) -- (g8);
\draw(d8) -- (h8);
\draw(e8) -- (h8);
\draw(f8) -- (j8);
\draw(g8) -- (j8);
\draw(d8) -- (i8);
\draw(e8) -- (i8);
\draw(f8) -- (i8);
\draw(g8) -- (i8);
\end{tikzpicture}  
\end{center}
\end{example}

It follows from the definition that for any $f\in\Inc^q(P)$, $\Pro(f)
\in\Inc^q(P)$; see \cite[Prop.\ 3.4]{dilks2019rowmotion}. This paper also gives another definition of promotion on $\Inc^q(P)$ in terms of \emph{Bender--Knuth involutions} (see \cite[Prop.\ 3.1]{dilks2019rowmotion}) and shows these two definitions agree. We will not need this alternate definition, but remark that knowing both characterizations of promotion often proves useful.

\Cref{def:ProJDT} invites the following.

\begin{definition}[\protect{\cite[Def.\ 3.6]{dilks2019rowmotion}}]
\label{def:SlidingSubPoset}
Given $f \in \Inc^q(P)$, let $S(f)$ be the induced subposet of $P$ consisting of elements which are labeled with $\boxempty$ at some point in the process of promotion. We call $S(f)$ the \emph{sliding subposet} of $f$.
\end{definition}

 The following new characterization of the sliding subposet will be useful in later proofs. 

\begin{lemma}\label{lem:AlternateDescriptionSliding}
Given $f \in \Inc^q(P)$, define $S_1(f) = \{x \in P: f(x) = 1\}$ and for $i > 1$, recursively define $S_i(f) = \{x \in P: \exists~ y \in S_1(f) \cup \cdots \cup S_{i-1}(f) \text{ such that } y \lessdot x \text{ and } f(x) = \min_{z \gtrdot y}\{f(z)\}\}$. Then $S(f)$ is the induced subposet of $P$ on $\bigcup_{i=1}^q S_i(f)$.
\end{lemma}

\begin{proof}
Set $S'(f) = \bigcup_{i=1}^q S_i(f)$. First, we show $S(f) \subseteq S'(f)$. Let $x \in S(f)$ and let $k$ be the number of elements in the longest saturated chain contained in $S(f)$ which contains both $x$ and a minimal element of $S(f)$. The definition of jeu de taquin promotion implies such a chain will exist. 

If $k = 1$, then $x$ is a minimal element in $P$. Then, in order for $x$ to be in $S(f)$, it must be that $f(x) = 1$, in which case $x \in S_1(f)$.

Now, suppose we have shown the claim for all $1 \leq j < k$, and consider $x \in S(f)$ such that all saturated chains consisting of elements of $S(f)$, with minimal element being minimal in $P$ and with $x$ as a  maximal element, have length at most $k$. The fact that $x \in S(f)$ means that at some point in the process of promotion, a box slides from some $y \lessdot x$, $y \in S(f)$ to $x$. Since we perform the jeu de taquin slides in increasing order, the fact that a box slides up from some $y \lessdot x$ to $x$ means that $f(x) = \min_{z \gtrdot y} f(z)$. Therefore, since $y \in S'(f)$ by the inductive hypothesis, $x \in S'(f)$ by definition.  

We can similarly show $S'(f) \subseteq S(f)$, so that we conclude these two sets are equal. 
\end{proof}

\begin{example}
We draw the sliding subposet $S(f)$ for the increasing labeling $f$ given in \Cref{ex:Pro}; elements and cover relations in $S(f)$ are drawn in bold black and other elements and cover relations are drawn in dotted gray.

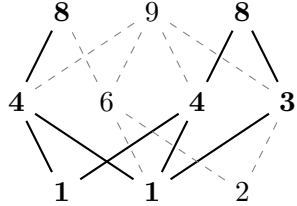
\begin{center}
\begin{tikzpicture}[scale=1.2]
\node(a) at (0.5,0){$\mathbf{1}$};
\node(b) at (1.5,0){$\mathbf{1}$};
\node(c) at (2.5,0){2};
\node(d) at (0,1){$\mathbf{4}$};
\node(e) at (1,1){6};
\node(f) at (2,1){$\mathbf{4}$};
\node(g) at (3,1){$\mathbf{3}$};
\node(h) at (0.5,2){$\mathbf{8}$};
\node(i) at (1.5,2){9};
\node(j) at (2.5,2){$\mathbf{8}$};
\draw[thick](a) -- (d);
\draw[thick](a) -- (f);
\draw[thick](b) -- (d);
\draw[gray,dashed](b) -- (e);
\draw[thick](b) -- (f);
\draw[thick](b) -- (g);
\draw[gray,dashed](c) -- (e);
\draw[gray,dashed](c) -- (g);
\draw[thick](d) -- (h);
\draw[gray,dashed](e) -- (h);
\draw[thick](f) -- (j);
\draw[thick](g) -- (j);
\draw[gray,dashed](d) -- (i);
\draw[gray,dashed](e) -- (i);
\draw[gray,dashed](f) -- (i);
\draw[gray,dashed](g) -- (i);
\end{tikzpicture}  
\end{center}

In the notation of \Cref{ex:FirstExOfIncrLabeling}, the layers of the sliding subposet are \[
S_1(f) = \{a,b\} \qquad S_2(f) = \{a,b,d,f,g\} \qquad S_3(f) = \{a,b,d,f,g,h,j\}.
\]
\end{example}

Increasing labelings under promotion exhibit a dynamical phenomenon called \emph{resonance} (see \cite{dilks2017resonance}), in which there is a map to another set such that promotion corresponds to rotation. That map is the binary content of the labeling, defined below.

\begin{definition}\label{def:Con}
Given $f \in \Inc^q(P)$, let the \emph{binary content word}, $\Con(f)$, be defined by $ (c_1,c_2,\ldots,c_q)$ where $c_i = 1$ if there exists $x \in P$ with $f(x) = i$ and otherwise $c_i = 0$.
\end{definition}

See \Cref{ex:ContentRotation} for examples of binary content words. 

\smallskip
Define the cyclic shift map $\rho:\mathbb{R}^k\rightarrow \mathbb{R}^k$ as $\rho((v_1, v_2, \ldots, v_k))=(v_2, v_3, \ldots, v_{k},v_1)$.

\begin{lemma}[\protect{\cite[Lemma 3.9]{dilks2019rowmotion}}]\label{lem:ProRotateContent}
If $f \in \Inc^q(P)$ then $\Con(\Pro(f))=\rho(\Con(f))$.
\end{lemma}

\begin{example}\label{ex:ContentRotation}
For an example of \Cref{lem:ProRotateContent}, notice that the content word of the labeling $f$ in \Cref{ex:FirstExOfIncrLabeling}  is $(1,1,1,1,0,1,0,1,1)$ and the content word of $\Pro(f)$, as in \Cref{ex:Pro}, is $(1,1,1,0,1,0,1,1,1)$.
\end{example}

We define the \emph{period} of a word to be the minimal number $\ell$ such that applying $\ell$ cyclic shifts of the word returns the same word. For example, the period of 0101 is 2 and the period of 0001 is~4. 

Now we look at a second piece of information which can be extracted from an increasing labeling and which has a nice relationship with promotion.

\begin{definition}
\label{def:deflation} 
Given $f \in \Inc^q(P)$, let   $\{i_1, i_2, \ldots, i_r\}$ be such that $i_j < i_k$ for $j < k$ and $f(P) = \{i_1,\ldots,i_r\}$.  We define the \emph{deflation} of $f$, denoted $\overline{f}$, to be the increasing labeling in $\Inc^r(P)$ such that for all $x \in P$, if $f(x) = i_j$, then $\overline{f}(x) = j$.

If $f = \overline{f}$, we will call $f$ \emph{packed}. If $f$ is packed and has values in $[r]$, where this is the smallest such $r$, we may also say $f$ is \emph{$r$-packed}.
\end{definition}

It is straightforward to see that $\overline{f}$ is also an increasing labeling of $P$. An example of deflation is given in \Cref{ex:Deflate}. 

The following lemma shows that the sliding subposets of an increasing labeling and its deflation coincide.
\begin{lemma}
\label{lem:sliding_deflation}
Let $f \in \Inc^q(P)$ be such that $f(x) = 1$ for some $x \in P$. Then, $S(f) = S(\bar{f})$.
\end{lemma}

\begin{proof}
We use the characterization of the sliding subposet from \Cref{lem:AlternateDescriptionSliding}, namely, $S(f) = \bigcup_{i=1}^q S_i(f)$. First, note that $S_1(f) = S_1(\bar{f})$ since we insist that $f$ labels at least one element of $P$ with a $1$. Then, $S_i(f) = S_i(\bar{f})$, since deflation preserves relative ordering of the elements.    
\end{proof}

The following proposition shows that promotion and deflation interact nicely.  This proposition was proven for minuscule posets (using only the assumption that there is a unique maximal element) in \cite[Proposition 5.1]{mandel2018orbits} (see also \cite[Sec.~3]{Pechenik22}) using the Bender--Knuth definition of promotion; we use a different proof method involving the sliding subposet to prove it for arbitrary posets. 

\begin{proposition}\label{prop:ProRestricts}
Let $f \in \Inc^q(P)$ and $\overline{f}$ be $r$-packed. If the first entry in $\Con(f)$ is~$1$, we have $\overline{\Pro(f)} = \Pro(\overline{f})$. Otherwise, we have $\overline{\Pro(f)} = \overline{f}$.
\end{proposition}

\begin{proof}

    First assume the first entry in $\Con(f)$ is 1. Let the labels used by $f$ be $\{1 = i_1 < i_2 < \cdots < i_r\}$ so that $\overline{f} \in \Inc^r(P)$. From \Cref{lem:ProRotateContent}, we know $\Pro(f)$ has labels $\{i_2-1 < i_3 -1 < \cdots < i_r -1 < q\}$.    
    Pick $x \in P$ and let $f(x) = i_j$ so that $\overline{f}(x) = j$.
    \begin{itemize}
        \item If $x \notin S(f)$, then $\Pro(f)(x) = f(x) - 1$ and $\overline{\Pro(f)}(x) = j-1$. Since $S(f) = S(\overline{f})$ by \Cref{lem:sliding_deflation}, $x \notin S(\overline{f})$, and $\Pro(\overline{f})(x) = j-1$ as well.  
        \item If $x \in S(f)$ and $x$ is maximal, then $\Pro(f)(x) = q$ and $\overline{\Pro(f)}(x) = r$. Again, $x \in S(\overline{f})$ as well, so $\Pro(\overline{f})(x) = r$.
        \item If $x \in S(f)$ and $x$ is not maximal, then $\Pro(f)(x) = \min_{z \gtrdot x} f(z) - 1$. Let $w \in P$ be an element satisfying this minimum, i.e. $w \gtrdot x$ and $f(w) = \min_{z \gtrdot x}f(z)$. Let $f(w) = i_k$. Then, $\overline{\Pro(f)}(x) = k-1$. Since deflation preserves relative orders, $k = \overline{f}(w) = \min_{z \gtrdot x}\overline{f}(z)$ as well. Therefore, $\Pro(\overline{f})(x) = k-1$.
    \end{itemize}
    We have at this point shown that for all $x \in P$, $\overline{\Pro(f)}(x) = \Pro(\overline{f})(x)$.

        Now assume the first entry in $\Con(f)$ is 0.  In this case, $\sigma_{1\to\boxempty}$ does not change $f$ because there are no $1$ labels in $f$.  Thus $\Pro(f)=f-1$, i.e. the result of decrementing $f$ by one at each element, and the conclusion holds.
\end{proof}

\begin{example}\label{ex:Deflate}
We demonstrate \Cref{prop:ProRestricts}  on our running example.

 
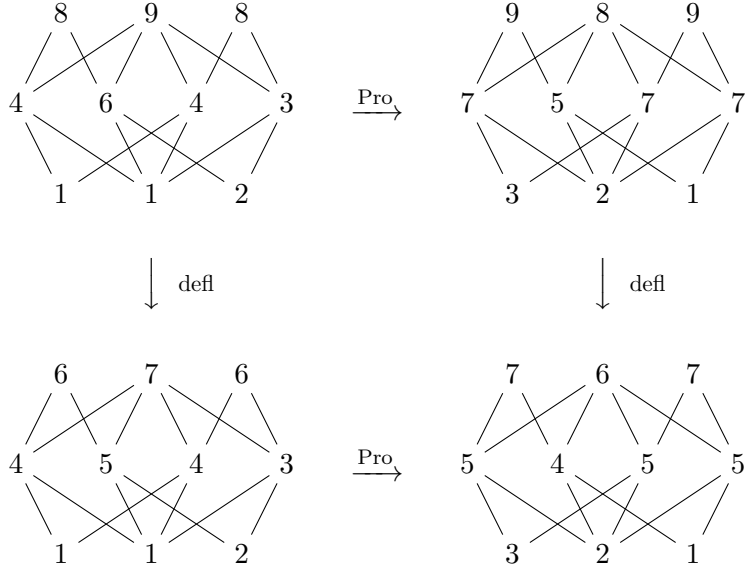
\begin{center}
\begin{tikzpicture}[scale=1.2]
\node(a) at (0.5,0){$1$};
\node(b) at (1.5,0){$1$};
\node(c) at (2.5,0){$2$};
\node(d) at (0,1){$4$};
\node(e) at (1,1){$6$};
\node(f) at (2,1){$4$};
\node(g) at (3,1){$3$};
\node(h) at (0.5,2){$8$};
\node(i) at (1.5,2){$9$};
\node(j) at (2.5,2){$8$};
\draw(a) -- (d);
\draw(a) -- (f);
\draw(b) -- (d);
\draw(b) -- (e);
\draw(b) -- (f);
\draw(b) -- (g);
\draw(c) -- (e);
\draw(c) -- (g);
\draw(d) -- (h);
\draw(e) -- (h);
\draw(f) -- (j);
\draw(g) -- (j);
\draw(d) -- (i);
\draw(e) -- (i);
\draw(f) -- (i);
\draw(g) -- (i);
\node() at (4,1){$\xrightarrow[]{\Pro}$};
\node() at (1.5,-1){$\xdownarrow{.45cm}$};
\node[scale=0.8] at (2,-1){defl};
\node(a1) at (5.5,0){$3$};
\node(b1) at (6.5,0){$2$};
\node(c1) at (7.5,0){$1$};
\node(d1) at (5,1){$7$};
\node(e1) at (6,1){$5$};
\node(f1) at (7,1){$7$};
\node(g1) at (8,1){$7$};
\node(h1) at (5.5,2){$9$};
\node(i1) at (6.5,2){$8$};
\node(j1) at (7.5,2){$9$};
\draw(a1) -- (d1);
\draw(a1) -- (f1);
\draw(b1) -- (d1);
\draw(b1) -- (e1);
\draw(b1) -- (f1);
\draw(b1) -- (g1);
\draw(c1) -- (e1);
\draw(c1) -- (g1);
\draw(d1) -- (h1);
\draw(e1) -- (h1);
\draw(f1) -- (j1);
\draw(g1) -- (j1);
\draw(d1) -- (i1);
\draw(e1) -- (i1);
\draw(f1) -- (i1);
\draw(g1) -- (i1);
\node() at (6.5,-1){$\xdownarrow{.45cm}$};
\node[scale=0.8] at (7,-1){defl};
\node(a2) at (0.5,-4){$1$};
\node(b2) at (1.5,-4){$1$};
\node(c2) at (2.5,-4){$2$};
\node(d2) at (0,-3){$4$};
\node(e2) at (1,-3){$5$};
\node(f2) at (2,-3){$4$};
\node(g2) at (3,-3){$3$};
\node(h2) at (0.5,-2){$6$};
\node(i2) at (1.5,-2){$7$};
\node(j2) at (2.5,-2){$6$};
\draw(a2) -- (d2);
\draw(a2) -- (f2);
\draw(b2) -- (d2);
\draw(b2) -- (e2);
\draw(b2) -- (f2);
\draw(b2) -- (g2);
\draw(c2) -- (e2);
\draw(c2) -- (g2);
\draw(d2) -- (h2);
\draw(e2) -- (h2);
\draw(f2) -- (j2);
\draw(g2) -- (j2);
\draw(d2) -- (i2);
\draw(e2) -- (i2);
\draw(f2) -- (i2);
\draw(g2) -- (i2);
\node() at (4,-3){$\xrightarrow[]{\Pro}$};
\node(a3) at (5.5,-4){$3$};
\node(b3) at (6.5,-4){$2$};
\node(c3) at (7.5,-4){$1$};
\node(d3) at (5,-3){$5$};
\node(e3) at (6,-3){$4$};
\node(f3) at (7,-3){$5$};
\node(g3) at (8,-3){$5$};
\node(h3) at (5.5,-2){$7$};
\node(i3) at (6.5,-2){$6$};
\node(j3) at (7.5,-2){$7$};
\draw(a3) -- (d3);
\draw(a3) -- (f3);
\draw(b3) -- (d3);
\draw(b3) -- (e3);
\draw(b3) -- (f3);
\draw(b3) -- (g3);
\draw(c3) -- (e3);
\draw(c3) -- (g3);
\draw(d3) -- (h3);
\draw(e3) -- (h3);
\draw(f3) -- (j3);
\draw(g3) -- (j3);
\draw(d3) -- (i3);
\draw(e3) -- (i3);
\draw(f3) -- (i3);
\draw(g3) -- (i3);
\end{tikzpicture}
 \end{center}

\end{example}

\Cref{prop:ProRestricts} immediately implies the following, which will be useful in later proofs.

\begin{corollary}\label{cor:can-find-lift}
    Let $P$ be a poset and $\mathcal{O}$ be a promotion orbit in $\Inc^q(P)$.  Suppose $f\in \mathcal{O}$ and $\overline{f}$ is $r$-packed.  Let $\overline{\mathcal{O}}$ be the promotion orbit of $\overline{f}$ in $\Inc^r(P)$. If $w\in \overline{\mathcal{O}}$, then there exists $g\in \mathcal{O}$ such that $\overline{g}=w$.
\end{corollary}

A consequence of the above proposition and corollary is that all labelings in a promotion orbit $\mathcal{O}$ of $\Inc^q(P)$ deflate to all labelings in a single promotion orbit $\overline{\mathcal{O}}$ of $\Inc^r(P)$.  Thus, by abuse of notation, we can refer to $\overline{\mathcal{O}}$ as the deflation of the orbit $\mathcal{O}$.

\begin{remark}\label{remark:defl_con}
Note that an increasing labeling is uniquely determined by its deflation and its binary content word (see \cite[Proposition 4.2]{mandel2018orbits}).  Therefore, if we want to understand the promotion orbit containing a fixed labeling, it suffices to consider the period of its binary content word and the orbit of its deflation.
\end{remark}

The analogue of the next theorem was proven for minuscule posets in \cite[Theorem 6.1]{mandel2018orbits}; given \Cref{prop:ProRestricts} above, the proof method of \cite[Theorem 6.1]{mandel2018orbits} yields the following. 
\begin{theorem}\label{thm:OrderComparedToDeflation}
Let $f \in \Inc^q(P)$ such that $\Con(f)$ has period $\ell$. Suppose $\overline{f}$ is $r$-packed and  
the size of the promotion orbit containing $\overline{f}$ is $\tau$. Then, the promotion orbit of $f$ has size \[
\frac{\tau \ell}{\gcd(r\ell/q,\tau)}.
\]
\end{theorem}
In the next section, we apply this theorem to find the order of promotion on specific posets. Then in \Cref{sec:swap,sec:deflation}, we use it to prove several instances of orthomesy.

We also have the following general corollary.
\begin{corollary}\label{cor:Divisible}
    For any poset $P$ and $q>|P|$, the order of promotion on $\Inc^q(P)$ is divisible by~$q$.
\end{corollary}

\begin{proof}
    Let $f$ be a packed increasing labeling of $P$.  We can consider $f$ to be an element of $\Inc^q(P)$. Then $\Con(f)=(1,1,\dots,1,0,0,\dots,0)$.  The period of $\Con(f)$ is $q$.  Thus, by \Cref{thm:OrderComparedToDeflation}, the promotion orbit of $f$ has size $\displaystyle\frac{\tau}{\gcd(r,\tau)}q$, where $\tau,r$ are as in the theorem statement.  Since $\displaystyle\frac{\tau}{\gcd(r,\tau)}$ is an integer, the conclusion holds.
\end{proof}

\section{Linear Growth of promotion}
\label{sec:lin_growth}

\Cref{thm:OrderComparedToDeflation} tells us that if we want to understand the order of promotion on  $\Inc^q(P)$, we only need to understand the orbits of $r$-packed increasing labelings for $r \leq q$. In particular, we can consider all $q \geq \vert P \vert$ simultaneously by considering all packed labelings of $P$. We apply this strategy to explicitly compute the order of promotion on small zig-zag posets. 
Recall  we denote as $\calZ_n$ the zig-zag poset with $n$ elements.

It is easy to see that the order of promotion on $\Inc^q(P)$ for $P=[n]$, the chain of $n$ elements, is $q$ for $q>n$. 
We now obtain the order of promotion on $\calZ_3$, the first zig-zag poset that is not a chain.

\begin{corollary}\label{thm:linear-Z3}
The order of promotion on increasing labelings of $\calZ_3$ with labels in $[q]$ for $q \geq 3$ is $2q$.
\end{corollary}
\begin{proof}
The poset
$\calZ_3$ has one $2$-packed labeling (with $1$ labels on both bottom elements and a $2$ label on the top) which maps to itself under promotion in $\Inc^2(\calZ_3)$. The period of $\Con(f)$ for $f$ that deflates to this is always $q$ if $q$ is odd and may be $q/2$ if $q$ is even. Thus by \Cref{thm:OrderComparedToDeflation} the promotion orbit size of this orbit in $\Inc^q(\calZ_3)$ is $1\ell/\gcd(2\cdot \ell/q,1)=\ell$, where $\ell=q$ or $q/2$.

There are two $3$-packed labelings (with a $1$ and a $2$ on the bottom and $3$ on the top) which map to each other under promotion in $\Inc^3(\calZ_3)$. $\Con(f)$ has three $1$'s and $q-3$ zeros, so its period may be $q$ or $q/3$, if $q$ is divisible by $3$. Thus by \Cref{thm:OrderComparedToDeflation} the promotion orbit size of this orbit in $\Inc^q(\calZ_3)$ is $2\ell/\gcd(3\cdot \ell/q,2)$. For each of $\ell=q,q/3$ this equals $2q$ or $2q/3$, respectively.

Since there are orbits of size $q$, $q/2$ (if $q$ is even), $2q$ (if $q>3$), and $2q/3$ (if $q$ is divisible by 3), the order of promotion is $2q$. 
\end{proof}

\begin{remark}
\Cref{thm:linear-Z3} may also be proved using a result of Plante and Roby \cite[Theorem 3.3]{PlanteRobyFPSAC2024} showing the order of \emph{rowmotion} on order ideals of the poset  $\calZ_3\times[q-2]$ is $2q$ combined with a bijection of Dilks, Striker, and Vorland \cite{dilks2019rowmotion} which specializes to show these order ideals under rowmotion are in equivariant bijection with the increasing labelings $\Inc^q(\calZ_3)$ under promotion.
This is also related to a more general result \cite[Theorem 4]{adenbaum2023order} on the order of \emph{$P$-strict labeling promotion}, as defined in \cite{bernstein2024p}.
\end{remark}

\begin{corollary}\label{thm:linear-Z4}
The order of promotion on increasing labelings of $\calZ_4$ with labels in $[q]$ for $q \geq 4$ is $15q$.
\end{corollary}

\begin{proof}
In \Cref{tab:Z4Orbits}, we describe all packed labelings of $\calZ_4$.
The first two rows tell us there is one packed 2-labeling; there are five 3-packed labelings which all are in the same promotion orbit; and there are five 4-packed labelings which split into promotion orbits of sizes two and three.

To use \Cref{thm:OrderComparedToDeflation}, we also need to consider all possible periods of the binary content words. If $\overline{f} \in \Inc^r(P)$, then the possible periods of $\Con(f)$ are exactly $\frac{q}{r'}$ where $r' \vert r$. 
The third to sixth rows of \Cref{tab:Z4Orbits} compute the quantity $\frac{\ell \tau}{\gcd(\ell r/q,\tau)}$ for all triples $(r,\tau,\ell)$ where $\ell = \frac{q}{r'}$ with $r' \vert r$.

The order of promotion on $\Inc^q(P)$ is the least common multiple of the sizes of all the orbits.  
If $q=4$, then the only orbits are of size $q$ and $q/2$ (where $r=2$), of size $5q$ (when $r=3$), and of size $q/2$ and $3q/4$ (when $r=4$).
So the order of promotion is $\lcm(q/2,3q/4,q,5q)=15q$.

Now consider the case where $q>4$.
In this case, the orbit types described in each row of \Cref{tab:Z4Orbits} arise exactly for the values of $q>4$ where the $\ell$ value for the row is an integer.  Since $\ell=q$ is always an integer, the orbit types in the first row always exist when $q>4$. 
The least common multiple of the orders listed in the first row is $15q$, so the order of promotion on $\Inc^q(P)$ is a multiple of $15q$.
Since \Cref{tab:Z4Orbits} addresses all possible orbits of $\Inc^q(P)$, the order of promotion on $\Inc^q(P)$ divides the least common multiple of the orders listed in \Cref{tab:Z4Orbits}, which is $15q$. 
Thus, the order of promotion on $\Inc^q(P)$ is exactly $15q$. 
\end{proof}

\begin{table}[htbp]
    \centering
    \renewcommand{\arraystretch}{1.5}
    \setlength{\tabcolsep}{10pt}
    \begin{tabular}{|c||c|c|c|c|}\hline
       $r$  & $2$ & $ 3$ & \multicolumn{2}{c|}{$4$}\\\hline
   $\tau$ & $1$ & $5$ & $2$ & $3$\\\hline\hline
$\ell = q$ & $q$ & $5q$ & $q$ & $3q$\\\hline
$\ell = \frac{q}{2}$ & $\frac{q}{2}$& & $\frac{q}{2}$ & $\frac{3q}{2}$\\\hline
$\ell = \frac{q}{3}$ & & $\frac{5q}{3}$ & & \\\hline
$\ell = \frac{q}{4}$ &  & & $\frac{q}{2}$ & $\frac{3q}{4}$\\\hline
    \end{tabular}
    \caption{All possible orbit sizes of the set of increasing labelings on $\calZ_4$ with values in $[q]$. Entries are left blank where it is not possible for a binary word with $r$ 1's to have period $\ell$.}
    \label{tab:Z4Orbits}
\end{table}

\begin{table}[hbtp]
    \centering
    \renewcommand{\arraystretch}{1.5}
    \setlength{\tabcolsep}{10pt}
    \begin{tabular}{|c||c|c|c|c|c|c|c|}\hline
        $r $ & $2$ & \multicolumn{2}{c|}{$3$} & \multicolumn{2}{c|}{$4$} & \multicolumn{2}{c|}{$5$}\\\hline
      $\tau $   & 1 & 2 & 8 & 2 & 10 & 4 & 12\\\hline\hline
    $\ell = q$ & $q$ & $2q$& $8q$ & $q$& $5q$& $4q$ & $12q$\\\hline
    $\ell = \frac{q}{2}$ & $\frac{q}{2}$ & & & $\frac{q}{2}$& $\frac{5q}{2}$& & \\\hline
    $\ell = \frac{q}{3}$ & & $\frac{2q}{3}$& $\frac{8q}{3}$& & & & \\\hline
    $\ell = \frac{q}{4}$ & & & &$\frac{q}{2}$  & $\frac{5q}{2}$ & & \\\hline
    $\ell = \frac{q}{5}$ & & & & & & $\frac{4q}{5}$ & $\frac{12q}{5}$\\\hline
    \end{tabular}
    \caption{All possible orbit sizes of the set of increasing labelings on $\calZ_5$ with values in $[q]$. Entries are left blank where it is not possible for a binary word with $r$ 1's to have period $\ell$.}
    \label{tab:Z5Orbits}
\end{table}

\begin{corollary}\label{thm:linear-Z5}
The order of promotion on increasing labelings of $\calZ_5$ with labels in $[q]$ for $q \geq 5$ is $120q$.
\end{corollary}

\begin{proof}
\Cref{tab:Z5Orbits} summarizes the promotion orbits of all packed labelings of $\calZ_5$. As in the proof of \Cref{thm:linear-Z4}, the orbit types in the first row always exist.  The least common multiple of the orders listed in the first row is $120q$, so the order of promotion on $\Inc^q(P)$ is a multiple of $120q$. The least common multiple of all the orders listed in \Cref{tab:Z5Orbits} is $120q$, so the order of promotion on $\Inc^q(P)$ divides $120q$. Thus, the order of promotion on $\Inc^q(P)$ is exactly $120q$.
\end{proof}

\begin{remark}\label{rmk:OrderOnZnLargern}
The procedure used for \Cref{thm:linear-Z4} and \Cref{thm:linear-Z5} works in principle for any fixed value of $n$. One may hope that the order of promotion for $\calZ_n$ would be of the form $m_{\calZ_n}\cdot q$ where $m_{\calZ_n}$ is a nice sequence. The order of promotion on increasing labelings of $\calZ_6$ with labels in $[q]$ for $q\geq 6$ computationally appears to be $13090q$ (tested up to $q=10$). This indicates that though characterizing the orbit sizes of packed labelings of $\calZ_n$ is a finite computation for fixed $n$, it is a difficult problem in general. See \Cref{sec:future} for further discussion.
\end{remark}

When $n$ is odd, we can improve \Cref{cor:Divisible} and show that the order of promotion on $\mathcal{Z}_n$ is divisible by $2q$. 

\begin{corollary}
For $n\geq3$ odd and $q > 3$, the order of promotion on $\Inc^q(\calZ_n)$ is divisible by $2q$. 
\end{corollary}

\begin{proof}
The following 3-packed labelings are in an orbit of size 2 together in $\Inc^3(\calZ_n)$:
\begin{center}
\begin{tikzpicture}
    \node (A) at (0,0) {1};
    \node (B) at (1,1) {3};
    \node (C) at (2,0) {2};
    \node (D) at (3,1) {3};
    \node (E) at (4,0) {1};
    \node (F) at (5,1) {3};
    \node (G) at (6,0) {2};
    \node at (7.5,0.5) {$\dots$};
    \draw (A) -- (B);
    \draw (B) -- (C);
    \draw (C) -- (D);
    \draw (D) -- (E);
    \draw (E) -- (F);
    \draw (F) -- (G);
    \draw (G) -- (6.5,0.5);
\end{tikzpicture}

\begin{tikzpicture}
    \node (A) at (0,0) {2};
    \node (B) at (1,1) {3};
    \node (C) at (2,0) {1};
    \node (D) at (3,1) {3};
    \node (E) at (4,0) {2};
    \node (F) at (5,1) {3};
    \node (G) at (6,0) {1};
    \node at (7.5,0.5) {$\dots$};
    \draw (A) -- (B);
    \draw (B) -- (C);
    \draw (C) -- (D);
    \draw (D) -- (E);
    \draw (E) -- (F);
    \draw (F) -- (G);
    \draw (G) -- (6.5,0.5);
\end{tikzpicture}
\end{center}

Let $f$ be the first of these labelings considered as an element of $\Inc^q(P)$ with $q>3$.  Then the period of $\Con(f)$ is $q$. By \Cref{thm:OrderComparedToDeflation}, the promotion orbit of $f$ has size $\frac{2q}{\gcd(3,2)}=2q$.
\end{proof}
\section{Orthomesy and Swap}
\label{sec:swap}
Throughout this section and the next, let $P$ be a self-dual poset that has an order-reversing involution $\kappa: P \to P$. In this section, we define and study an operation called \emph{swap} on increasing labelings of a self-dual poset. We prove a homomesy for $\swap$ in \Cref{prop:swap-homomesy-antipodal}, which we then use to exhibit a class of orthomesies in \Cref{thm:swap_orthomesy}. The definition of $\swap$ is inspired by the  \emph{ideal complement} studied in \cite[p.4]{elizalde2021rowmotion}.  (Indeed, $\swap$ specializes to  ideal complement when $q=2$.)

\begin{definition}
  Let $P$ be a self-dual poset which has an order-reversing involution $\kappa:P\to P$.  Given $f \in \Inc^q(P)$, define $\swap(f)(x) = q+1 - f(\kappa(x))$.  
\end{definition}

See \Cref{ex:SwapThings} for an example of the swap operation on an increasing labeling.  We next consider how the swap operation interacts with the operations defined in \Cref{sec:PromotionBasics}.
The first is immediate from the definitions.

\begin{lemma}\label{lem:SwapOnContent}
Given $f \in \Inc^q(P)$, if $\Con(f) = (c_1,\ldots,c_q)$, then $\Con(\swap(f)) = (c_q,\ldots,c_1)$.
\end{lemma}

When we use this lemma in the future it will be helpful to have the following notation: given a vector $v$ we define $\rev(v)$ to be the vector written in reverse order.  Thus, \Cref{lem:SwapOnContent} says $\Con(\swap(f))=\rev(\Con(f))$. For instance, given the increasing labeling $f$ as in \Cref{ex:FirstExOfIncrLabeling}, we see that $\Con(f) = (1,1,1,1,0,1,0,1,1)$ and $\Con(\swap(f)) = (1,1,0,1,0,1,1,1,1) = \rev((1,1,1,1,0,1,0,1,1))$.

Recall from \Cref{def:deflation} that $\overline{f}$ denotes the deflation of $f$.

\begin{lemma}\label{lem:SwapAndDeflation}
Let $f \in \Inc^q(P)$, and suppose $\overline{f}\in\Inc^r(P)$ is $r$-packed.  Then $\overline{\swap(f)} = \swap(\overline{f})$.
\end{lemma}

\begin{proof}
Let $x \in P$ and $\overline{f}(\kappa(x)) = i$. Suppose that $\overline{f} \in \Inc^r(P)$. Then, by definition, $\swap(\overline{f})(x) = r+1-i$. 

Since $\overline{f} \in \Inc^r(P)$, we know that there are $r$ distinct labels used by $f$, and that $\kappa(x)$ is labeled with the $i$-th smallest label. Equivalently, $\swap(f)$ labels $x$  with the $i$-th largest number amongst the labels in $\swap(f)$. Therefore, $\overline{\swap(f)}(x) = r+1-i$ as well.
\end{proof}

\begin{example}\label{ex:SwapThings}
The poset in \Cref{ex:FirstExOfIncrLabeling} is self dual and has the following  order-reversing involution.

\begin{center}
\begin{tabular}{c||*{10}{c|}}
$x$ & $a$ & $b$ & $c$ & $d$ & $e$ & $f$ & $g$ & $h$ & $i$ & $j$ \\\hline
$\kappa(x)$ & $h$ & $i$ & $j$ & $d$ & $f$ & $e$ & $g$ &  $a$ & $b$ & $c$\\
\end{tabular}
\end{center}

Using this involution, we depict in the top row of the diagram below $\swap(f)$ for $f$ the increasing labeling given in \Cref{ex:FirstExOfIncrLabeling}, viewing $f$ as an element of $\Inc^9(P)$.
We furthermore demonstrate  the relation of \Cref{lem:SwapAndDeflation} between swap and deflation for the increasing labeling $f$.

\begin{center}
\begin{tikzpicture}[scale=1.2]
\node(a) at (0.5,0){$1$};
\node(b) at (1.5,0){$1$};
\node(c) at (2.5,0){$2$};
\node(d) at (0,1){$4$};
\node(e) at (1,1){$6$};
\node(f) at (2,1){$4$};
\node(g) at (3,1){$3$};
\node(h) at (0.5,2){$8$};
\node(i) at (1.5,2){$9$};
\node(j) at (2.5,2){$8$};
\draw(a) -- (d);
\draw(a) -- (f);
\draw(b) -- (d);
\draw(b) -- (e);
\draw(b) -- (f);
\draw(b) -- (g);
\draw(c) -- (e);
\draw(c) -- (g);
\draw(d) -- (h);
\draw(e) -- (h);
\draw(f) -- (j);
\draw(g) -- (j);
\draw(d) -- (i);
\draw(e) -- (i);
\draw(f) -- (i);
\draw(g) -- (i);
\node() at (4,1){$\xrightarrow[]{\swap}$};
\node() at (1.5,-1){$\xdownarrow{.45cm}$};
\node[scale=0.8] at (2,-1){defl};
\node(a1) at (5.5,0){$2$};
\node(b1) at (6.5,0){$1$};
\node(c1) at (7.5,0){$2$};
\node(d1) at (5,1){$6$};
\node(e1) at (6,1){$6$};
\node(f1) at (7,1){$4$};
\node(g1) at (8,1){$7$};
\node(h1) at (5.5,2){$9$};
\node(i1) at (6.5,2){$9$};
\node(j1) at (7.5,2){$8$};
\draw(a1) -- (d1);
\draw(a1) -- (f1);
\draw(b1) -- (d1);
\draw(b1) -- (e1);
\draw(b1) -- (f1);
\draw(b1) -- (g1);
\draw(c1) -- (e1);
\draw(c1) -- (g1);
\draw(d1) -- (h1);
\draw(e1) -- (h1);
\draw(f1) -- (j1);
\draw(g1) -- (j1);
\draw(d1) -- (i1);
\draw(e1) -- (i1);
\draw(f1) -- (i1);
\draw(g1) -- (i1);
\node() at (6.5,-1){$\xdownarrow{.45cm}$};
\node[scale=0.8] at (7,-1){defl};
\node(a2) at (0.5,-4){$1$};
\node(b2) at (1.5,-4){$1$};
\node(c2) at (2.5,-4){$2$};
\node(d2) at (0,-3){$4$};
\node(e2) at (1,-3){$5$};
\node(f2) at (2,-3){$4$};
\node(g2) at (3,-3){$3$};
\node(h2) at (0.5,-2){$6$};
\node(i2) at (1.5,-2){$7$};
\node(j2) at (2.5,-2){$6$};
\draw(a2) -- (d2);
\draw(a2) -- (f2);
\draw(b2) -- (d2);
\draw(b2) -- (e2);
\draw(b2) -- (f2);
\draw(b2) -- (g2);
\draw(c2) -- (e2);
\draw(c2) -- (g2);
\draw(d2) -- (h2);
\draw(e2) -- (h2);
\draw(f2) -- (j2);
\draw(g2) -- (j2);
\draw(d2) -- (i2);
\draw(e2) -- (i2);
\draw(f2) -- (i2);
\draw(g2) -- (i2);
\node() at (4,-3){$\xrightarrow[]{\swap}$};
\node(a3) at (5.5,-4){$2$};
\node(b3) at (6.5,-4){$1$};
\node(c3) at (7.5,-4){$2$};
\node(d3) at (5,-3){$4$};
\node(e3) at (6,-3){$4$};
\node(f3) at (7,-3){$3$};
\node(g3) at (8,-3){$5$};
\node(h3) at (5.5,-2){$7$};
\node(i3) at (6.5,-2){$7$};
\node(j3) at (7.5,-2){$6$};
\draw(a3) -- (d3);
\draw(a3) -- (f3);
\draw(b3) -- (d3);
\draw(b3) -- (e3);
\draw(b3) -- (f3);
\draw(b3) -- (g3);
\draw(c3) -- (e3);
\draw(c3) -- (g3);
\draw(d3) -- (h3);
\draw(e3) -- (h3);
\draw(f3) -- (j3);
\draw(g3) -- (j3);
\draw(d3) -- (i3);
\draw(e3) -- (i3);
\draw(f3) -- (i3);
\draw(g3) -- (i3);
\end{tikzpicture}
 \end{center}

\end{example}

Recall the sliding subposet from \Cref{def:SlidingSubPoset}.
Define $S^{-1}(f)$ to be the sliding subposet of $\Pro^{-1}(f)$. 
The following lemma is analogous to \Cref{lem:AlternateDescriptionSliding} for $S^{-1}$. Its proof is also analogous, so we omit it. Note this lemma, like \Cref{lem:AlternateDescriptionSliding}, holds for all posets, not only self-dual posets.

\begin{lemma}\label{lem:AlternateDescriptionInverseSliding}
Given $f \in \Inc^q(P)$, define $S_1^{-1}(f) = \{x \in P: f(x) = q\}$ and for $i > 1$, recursively define $S_i^{-1}(f) = \{x \in P: \exists y \in S_1^{-1}(f) \cup \cdots \cup S_{i-1}^{-1}(f) \text{ such that } y \gtrdot x \text{ and } f(x) = \max_{y \gtrdot z}\{f(z)\}\}$.  Then $S^{-1}(f)$ is the induced poset of $P$ on $\bigcup_{i=1}^q S_i^{-1}(f)$.
\end{lemma}

The following is the key technical result to show that $\swap$ and $\Pro$ anticommute.

\begin{lemma}\label{lem:SlidingAndSwap}
Let $f \in \Inc^q(P)$. Then, $S(\swap(f)) = \kappa(S^{-1}(f))$. 
\end{lemma}

\begin{proof}
We will first prove a stronger statement, namely, for all $1 \leq i \leq q$, $S_i(\swap(f)) = \kappa(S^{-1}_i(f))$. The lemma then follows by \Cref{lem:AlternateDescriptionSliding} and \Cref{lem:AlternateDescriptionInverseSliding}.

First, we set $i = 1$. Following the definitions, \[
x \in S_1(\swap(f)) \iff \swap(f)(x) = 1 \iff f(\kappa(x)) = q \iff \kappa(x) \in S_1^{-1}(f). 
\]
Since $\kappa$ is an involution,  $\kappa(x) \in S_1^{-1}(f)$  if and only if $x \in \kappa(S_1^{-1}(f))$. 

Now, let $i > 1$ and assume we have shown the claim for all $j < i$. 
We have $x \in S_i(\swap(f))$ if and only if there exists $y \lessdot x$ such that $y \in \cup_{j =1}^{i-1} S(\swap(f))$ and $\swap(f)(x) = \min_{y \lessdot z} \swap(f)(z) = \min_{y \lessdot z} q+1 - f(\kappa(z)) = q+1 - \max_{y \lessdot z} f(\kappa(z))$.
By definition of $\swap$, we see that this equality is equivalent to \[
f(\kappa(x)) = \max_{y \lessdot z} f(\kappa(z)) = \max_{y \lessdot \kappa(z)} f(z) = \max_{\kappa(y) \gtrdot z}f(z)
\] where in the second equality we swap the roles of $z$ and $\kappa(z)$ and in the third we  use the fact that  the set of $\kappa(z)$ covering $y$ is in bijection with the set of $z$ covered by $\kappa(y)$ since $P$ is self-dual and $\kappa$ is an involution. 

At this point we have shown  $x \in S_i(\swap(f))$ if and only if there exists $y \lessdot x$ such that $y \in \cup_{j =1}^{i-1} S_j(\swap(f))$ and $f(\kappa(x)) = \max_{\kappa(y) \gtrdot z}f(z)$.
By our inductive hypothesis, $y \in \cup_{j=1}^{i-1} \kappa(S^{-1}_j(f))$, and equivalently $\kappa(y) \in \cup_{j=1}^{i-1} S^{-1}_j(f)$. Therefore, since $\kappa(y) \gtrdot \kappa(x)$, if $f(\kappa(x)) = \max_{\kappa(y) \gtrdot z}f(z)$, then $\kappa(x) \in S^{-1}_i(f)$. In conclusion, we have shown that $x \in S_i(\swap(f))$ if and only if $\kappa(x) \in S^{-1}_i(f)$, which is true if and only if $x \in \kappa(S^{-1}_i(f))$.
\end{proof}

\begin{proposition}\label{thm:SwapAndPro}
Given $f \in \Inc^q(P)$, $\swap(\Pro^{-1}(f)) =\Pro(\swap(f))$. 
\end{proposition}

\begin{proof}
Let $x \in P$. There are three possibilities for $x$: $x \in S(\swap(f))$ and $x$ is maximal, $x \in S(\swap(f))$ and $x$ is not maximal, or $x \notin S(\swap(f))$. For each case, we show $\Pro(\swap(f))(x) = \swap(\Pro^{-1}(f))(x)$.

First, let $x \in S(\swap(f))$ and suppose $x$ is maximal. Then, $\Pro(\swap(f))(x) = q$. If $x \in S(\swap(f))$ then, by \Cref{lem:SlidingAndSwap}, $\kappa(x) \in S^{-1}(f)$, and if $x$ is maximal then $\kappa(x)$ is minimal. Therefore, $\Pro^{-1}(f)(\kappa(x)) = 1$ and by definition $\swap(\Pro^{-1})(f)(x) = q+1-\Pro^{-1}(f)(\kappa(x)) = q = \Pro(\swap(f))(x)$.

Next, if $x \in S(\swap(f))$ and $x$ is not maximal, then $\Pro(\swap(f))(x) = \min_{z \gtrdot x} \swap(f)(z) - 1 = q - \max_{z \gtrdot x} f(\kappa(z)) = q - \max_{z \lessdot \kappa(x)} f(z)$. From \Cref{lem:SlidingAndSwap}, we know that if $x \in S(\swap(f))$ and $x$ is not maximal, then $\kappa(x) \in S^{-1}(f)$ and $\kappa(x)$ is not minimal. Therefore, $\swap(\Pro^{-1}(f))(x) = q+1-\Pro^{-1}(f)(\kappa(x)) = q +1 - (\max_{z \lessdot \kappa(x)} f(z) + 1) = \Pro(\swap(f))(x)$ as desired.

Finally, suppose $x \notin S(\swap(f))$, so that equivalently $x \notin \kappa(S^{-1}(f))$ by \Cref{lem:SlidingAndSwap} and $\kappa(x) \notin S^{-1}(f)$. We know $\Pro(\swap(f))(x) = \swap(f)(x) - 1 = q + 1 - f(\kappa(x)) - 1= q+1-(f(\kappa(x)) + 1) = q+1-\Pro^{-1}(f)(\kappa(x)) = \swap(\Pro^{-1}(f))(x)$.

We have shown $\Pro(\swap(f))(x) = \swap(\Pro^{-1}(f))(x)$ for each $x \in P$, so that these are the same labelings of $P$.
\end{proof}

\begin{example}
 We demonstrate \Cref{thm:SwapAndPro} for the running example of an increasing labeling $f$, as given in \Cref{ex:FirstExOfIncrLabeling}. 

\begin{center}
\begin{tikzpicture}[scale=1.2]
\node(a) at (0.5,0){$1$};
\node(b) at (1.5,0){$1$};
\node(c) at (2.5,0){$2$};
\node(d) at (0,1){$4$};
\node(e) at (1,1){$6$};
\node(f) at (2,1){$4$};
\node(g) at (3,1){$3$};
\node(h) at (0.5,2){$8$};
\node(i) at (1.5,2){$9$};
\node(j) at (2.5,2){$8$};
\draw(a) -- (d);
\draw(a) -- (f);
\draw(b) -- (d);
\draw(b) -- (e);
\draw(b) -- (f);
\draw(b) -- (g);
\draw(c) -- (e);
\draw(c) -- (g);
\draw(d) -- (h);
\draw(e) -- (h);
\draw(f) -- (j);
\draw(g) -- (j);
\draw(d) -- (i);
\draw(e) -- (i);
\draw(f) -- (i);
\draw(g) -- (i);
\node() at (4,1){$\xrightarrow[]{\Pro^{-1}}$};
\node() at (1.5,-1){$\xdownarrow{.45cm}$};
\node[scale=0.8] at (2,-1){$\swap$};
\node(a1) at (5.5,0){$2$};
\node(b1) at (6.5,0){$2$};
\node(c1) at (7.5,0){$1$};
\node(d1) at (5,1){$5$};
\node(e1) at (6,1){$3$};
\node(f1) at (7,1){$5$};
\node(g1) at (8,1){$4$};
\node(h1) at (5.5,2){$9$};
\node(i1) at (6.5,2){$7$};
\node(j1) at (7.5,2){$9$};
\draw(a1) -- (d1);
\draw(a1) -- (f1);
\draw(b1) -- (d1);
\draw(b1) -- (e1);
\draw(b1) -- (f1);
\draw(b1) -- (g1);
\draw(c1) -- (e1);
\draw(c1) -- (g1);
\draw(d1) -- (h1);
\draw(e1) -- (h1);
\draw(f1) -- (j1);
\draw(g1) -- (j1);
\draw(d1) -- (i1);
\draw(e1) -- (i1);
\draw(f1) -- (i1);
\draw(g1) -- (i1);
\node() at (6.5,-1){$\xdownarrow{.45cm}$};
\node[scale=0.8] at (7,-1){$\swap$};
\node(a2) at (0.5,-4){$2$};
\node(b2) at (1.5,-4){$1$};
\node(c2) at (2.5,-4){$2$};
\node(d2) at (0,-3){$6$};
\node(e2) at (1,-3){$6$};
\node(f2) at (2,-3){$4$};
\node(g2) at (3,-3){$7$};
\node(h2) at (0.5,-2){$9$};
\node(i2) at (1.5,-2){$9$};
\node(j2) at (2.5,-2){$8$};
\draw(a2) -- (d2);
\draw(a2) -- (f2);
\draw(b2) -- (d2);
\draw(b2) -- (e2);
\draw(b2) -- (f2);
\draw(b2) -- (g2);
\draw(c2) -- (e2);
\draw(c2) -- (g2);
\draw(d2) -- (h2);
\draw(e2) -- (h2);
\draw(f2) -- (j2);
\draw(g2) -- (j2);
\draw(d2) -- (i2);
\draw(e2) -- (i2);
\draw(f2) -- (i2);
\draw(g2) -- (i2);
\node() at (4,-3){$\xrightarrow[]{\Pro}$};
\node(a3) at (5.5,-4){$1$};
\node(b3) at (6.5,-4){$3$};
\node(c3) at (7.5,-4){$1$};
\node(d3) at (5,-3){$5$};
\node(e3) at (6,-3){$5$};
\node(f3) at (7,-3){$7$};
\node(g3) at (8,-3){$6$};
\node(h3) at (5.5,-2){$8$};
\node(i3) at (6.5,-2){$8$};
\node(j3) at (7.5,-2){$9$};
\draw(a3) -- (d3);
\draw(a3) -- (f3);
\draw(b3) -- (d3);
\draw(b3) -- (e3);
\draw(b3) -- (f3);
\draw(b3) -- (g3);
\draw(c3) -- (e3);
\draw(c3) -- (g3);
\draw(d3) -- (h3);
\draw(e3) -- (h3);
\draw(f3) -- (j3);
\draw(g3) -- (j3);
\draw(d3) -- (i3);
\draw(e3) -- (i3);
\draw(f3) -- (i3);
\draw(g3) -- (i3);
\end{tikzpicture}
 \end{center}
\end{example}

\begin{corollary}\label{cor:TwoTypesOfOrbits}
For each orbit $\mathcal{O}$ of $\Pro$ on $\Inc^q(P)$, there exists an orbit $\mathcal{O}'$ such that for all $f \in \mathcal{O}$, $\swap(f) \in \mathcal{O}'$. 
\end{corollary}

\begin{proof}
 Pick $f \in \mathcal{O}$ and let $\mathcal{O}'$ be such that $\swap(f) \in \mathcal{O}'$. Let $g \in \mathcal{O}$. Then, $g = \Pro^{n}(f)$ for some $n \geq 0$. By repeated application of \Cref{thm:SwapAndPro}, we can show \[
\swap(g) = \swap(\Pro^{n}(f)) = \Pro^{-n}(\swap(f))
\]
and conclude that $\swap(g) \in \mathcal{O}'$.
\end{proof}

 The previous result invites the following notation. Given a promotion orbit $\mathcal{O}$, let $\swap(\mathcal{O})$ denote the unique promotion orbit such that all $f \in \mathcal{O}$ have $\swap(f) \in \swap(\mathcal{O})$. We will say an orbit $\mathcal{O}$ is \emph{swap-closed} if $\swap(\mathcal{O}) = \mathcal{O}$.

We will now show $\swap$ is homomesic with respect to the antipodal sum and total sum statistics.

\begin{definition}\label{def:stats}
    Let $P$ be a self-dual poset which has an order-reversing involution $\kappa:P\to P$. Given a labeling $f$ on $P$ and an element $x\in P$, we define the \emph{antipodal sum statistic with respect to $x$}: $\as_x(f)=f(x)+f(\kappa(x))$.  We also define the \emph{total sum statistic} $\tot(f)=\sum_{x\in P} f(x)$.
\end{definition}

\begin{proposition}\label{prop:swap-homomesy-antipodal}
The action of $\swap$ on $\Inc^q(P)$ is homomesic with respect to the antipodal sum statistic $\as_x$ for all $x\in P$.
\end{proposition}

\begin{proof}
We show that every orbit of swap has the same average value, $q+1$.
Observe that orbits are of either size one or two.
When an orbit of $f\in\Inc^q(P)$ has size two, then 
\begin{align*}
\as_x(f) + \as_x(\swap(f))&= f(x) + f(\kappa(x)) + (q+1-f(\kappa(x))) + (q+1 - f(x)) \\
&= 2(q+1).
\end{align*}
Thus, the average value of the orbit is $q+1$.

When the orbit of $f\in\Inc^q(P)$ has size one, this means that $f(x) = q+1 - f(\kappa(x))$, thus $\as_x(f) = f(x) + f(\kappa(x)) = q+1 - f(\kappa(x)) + f(\kappa(x)) = q+1$, as desired.
\end{proof}

\begin{corollary}\label{prop:swap-homomesy-total}
Let $f\in \Inc^q(P)$.
The total sum statistic $\tot(f)$ is homomesic with respect to the action of swap.
\end{corollary}
\begin{proof}
This follows immediately from the previous proposition.
\end{proof}

We now have our first examples of orthomesy for increasing labeling promotion.

\begin{theorem}
\label{thm:swap_orthomesy}
Let $\mathcal{O}$ be a swap-closed promotion orbit of elements of $\Inc^q(P)$. Then, $\mathcal{O}$ exhibits orthomesy with respect to the antipodal sum statistic $\mathcal{A}_x$ for all $x \in P$ and with respect to the total sum statistic $\tot$.
\end{theorem}

\begin{proof}
    The statement for $\mathcal{A}_x$ follows from \Cref{prop:swap-homomesy-antipodal} and \Cref{prop:homomesy-to-orthomesy}.  The statement for $\tot$ follows from \Cref{prop:swap-homomesy-total} and \Cref{prop:homomesy-to-orthomesy}.
\end{proof}

\begin{corollary}\label{cor:swap-deflation-orthomesy}
   Let $f\in\Inc^q(P)$ and let $\mathcal{O}$ be the promotion orbit of $f$ in $\Inc^q(P)$.  Let $r$ be the largest label in $\overline{f}$ and $\overline{\mathcal{O}}$ be the orbit of $\overline{f}\in\Inc^r(P)$.  Suppose the following hold:
    \begin{itemize}
        \item $\rho^i(\Con(f))=\rev(\Con(f))$ for some $i\in[q]$.
        \item $|\mathcal{O}|=|\overline{\mathcal{O}}| \ell$ where $\ell$ is the period of $\Con(f)$.
        \item $\overline{\mathcal{O}}$ is swap-closed.
    \end{itemize}
    Then the orbit $\mathcal{O}$ is orthomesic with respect to the antipodal sum statistic $\as_x$ for all $x\in P$ and the total sum statistic $\tot$.
\end{corollary}

\begin{proof}
    Because $\overline{\mathcal{O}}$ is swap-closed, $\swap(\overline{f})\in\overline{\mathcal{O}}$.  Since $|\mathcal{O}|=|\overline{\mathcal{O}}| \ell$, we know that if we look at the deflations and content words of all the elements of $\mathcal{O}$, we get every pairing of elements in $\overline{\mathcal{O}}$ and rotations of $\Con(f)$.  In particular, $\mathcal{O}$ contains an element $g$ with deflation $\swap(\overline{f})$ and content $\rho^i(\Con(f))$.  By \Cref{lem:SwapOnContent},  since $\rho^i(\Con(f)) = \rev(\Con(f))$, $g$ has the same content as $\swap(f)$. Also, by \Cref{lem:SwapAndDeflation}, $g$ has the same deflation as $\swap(f)$.  Thus, from \Cref{remark:defl_con}, we have $g=\swap(f)$.  Therefore, $\mathcal{O}$ is swap-closed and by \Cref{thm:swap_orthomesy}, both the $\as_x$ and $\tot$ statistics are orthomesic on $\mathcal{O}$.
\end{proof}

\section{Orthomesy and Deflation}\label{sec:deflation}

In \Cref{cor:swap-deflation-orthomesy} we saw that we could prove an increasing labeling promotion orbit was orthomesic using information about the deflation of the orbit.  In this section we show some other instances where an orbit's deflation allows us to prove orthomesy.

We begin with a lemma about the global average of the statistics $\as_x$ and $\tot$.

\begin{lemma}\label{lem:GlobalAverage}
Let $P$ be a self-dual poset with an order-reversing involution and consider all elements of $\Inc^q(P)$. 
\begin{enumerate}
    \item For all $x \in P$, the global average of the antipodal sum statistic with respect to $x$ is $q+1$. That is,  \[
\frac{1}{\vert \Inc^q(P)\vert}\sum_{f \in \Inc^q(P)} \as_x(f) = q+1.
    \]
    \item The global average of the total sum statistic is   $\frac{\vert P \vert}{2} (q+1)$. That is, 
\[
\frac{1}{\vert \Inc^q(P)\vert} \sum_{f \in \Inc^q(P)} \tot(f) = \frac{\vert P \vert}{2}(q+1).
    \]
\end{enumerate}
\end{lemma}

\begin{proof}
From the proof of \Cref{prop:swap-homomesy-antipodal} we know that the average of $\as_x$ is $q
+1$ on every orbit of $\swap$ on $\Inc^q(P)$.  This means that the average of $\as_x$ on $\Inc^q(P)$ is $q+1$.  The second statement follows from the first statement and the fact that for all $f\in\Inc^q(P)$, $\sum_{x\in P}\as_x(f)=2\tot(f)$.
\end{proof}

\begin{proposition}\label{linear_ext} Given a self-dual poset $P$ with an order-reversing involution, let $f \in \Inc^q(P)$ such that $\overline{f}$ is a linear extension. If $\mathcal{O}$ is the promotion orbit of $\Inc^q(P)$ containing $f$, then the orbit $\mathcal{O}$ is orthomesic with respect to the total sum statistic $\tot$.
\end{proposition}

\begin{proof} Let $f$ have content word $\Con(f)$. Given $r,\ell,\tau$ as in \Cref{thm:OrderComparedToDeflation}, let $d = \gcd(r\ell/q,\tau)$.
Each cyclic rotation of $\Con(f)$ appears $\frac{\tau}{d}$ times when we view elements of $\mathcal{O}$ as pairs of deflations and content words. 
Since $\overline{f}$ is a linear extension, we in particular know that every label in $f$ is distinct and by \Cref{lem:ProRotateContent} the same is true for all $\Pro^i(f) \in \mathcal{O}$. Therefore, computing the total sum of $f$ consists of just adding the positions of the 1's in $\Con(f)$. If we sum the vectors $\Con(f) + \rho(\Con(f)) + \cdots + \rho^{\ell-1}(\Con(f))$, we have $\frac{r \ell}{q}(1,1,\ldots,1)$, so that the contribution to total sum is $\frac{r \ell}{q}{q+1 \choose 2}$. Therefore, $\sum_{f \in \mathcal{O}} \tot(f) = \frac{r \ell \tau (q+1)}{2d}$.
By \Cref{lem:ProRotateContent}, the size of the orbit is $\frac{\ell\tau}{d}$ so the average total sum is $\frac{r(q+1)}{2}$. Since $r$ must be $\vert P \vert$, by \Cref{lem:GlobalAverage}, this is equal to the global average. In other words, this orbit is orthomesic.
\end{proof}

\begin{example}\label{ex:linear_ext}
We demonstrate \Cref{linear_ext} with the following increasing labeling $f$ in $\Inc^6(\calZ_4)$:
\begin{center}
\begin{tikzpicture}
    \node (A) at (0,0) {1};
    \node (B) at (1,1) {6};
    \node (C) at (2,0) {2};
    \node (D) at (3,1) {4};
    \draw (A) -- (B);
    \draw (B) -- (C);
    \draw (C) -- (D);
\end{tikzpicture}
\end{center}

The deflation of $f$ is:
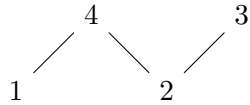
\begin{center}
\begin{tikzpicture}
   \node (A) at (0,0) {1};
    \node (B) at (1,1) {4};
    \node (C) at (2,0) {2};
    \node (D) at (3,1) {3};
    \draw (A) -- (B);
    \draw (B) -- (C);
    \draw (C) -- (D);
\end{tikzpicture}
\end{center}
Observe that $\overline{f}$ is a linear extension with largest label $4$. The promotion orbit of $\overline{f}$ in $\Inc^4(\calZ_4)$ is as follows:
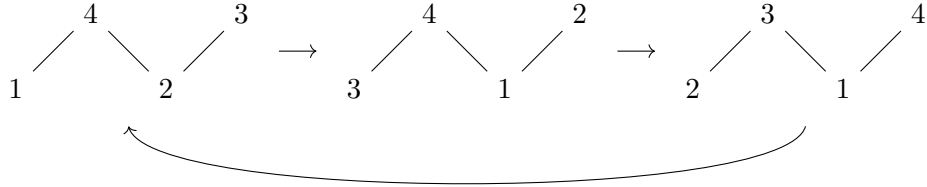
\begin{center}
\begin{tikzpicture}
   \node (A1) at (0,0) {1};
    \node (B1) at (1,1) {4};
    \node (C1) at (2,0) {2};
    \node (D1) at (3,1) {3};
    \draw (A1) -- (B1);
    \draw (B1) -- (C1);
    \draw (C1) -- (D1);
    \draw[->] (3.5, 0.5) -- (4,0.5);
   \node (A2) at (4.5,0) {3};
    \node (B2) at (5.5,1) {4};
    \node (C2) at (6.5,0) {1};
    \node (D2) at (7.5,1) {2};
    \draw (A2) -- (B2);
    \draw (B2) -- (C2);
    \draw (C2) -- (D2);
    \draw[->] (8, 0.5) -- (8.5,0.5);
   \node (A3) at (9,0) {2};
    \node (B3) at (10,1) {3};
    \node (C3) at (11,0) {1};
    \node (D3) at (12,1) {4};
    \draw (A3) -- (B3);
    \draw (B3) -- (C3);
    \draw (C3) -- (D3);
    \draw (10.5, -0.5) edge[->,bend left=75, looseness=0.3] (1.5,-0.5);
\end{tikzpicture}
\end{center}

We can see that the promotion orbit of $\overline{f}$ has 3 elements. Also, the period of the content word of $f$ is equal to $6$.
Thus, $\gcd(r\ell/q,\tau) = \gcd((4\cdot 6)/6, 3) = 1$.
We conclude by \Cref{linear_ext} that the promotion orbit of $f$ is orthomesic with respect to the total sum statistic. 
\end{example}


Observe that if $f\in \Inc^r(P)$ is both packed and a linear extension, then $r=|P|$.
Then the orbit average of the total sum statistic is $\frac{1}{|\mathcal{O}|}\left(\binom{r+1}{2} \cdot |\mathcal{O}|\right) = \frac{(r+1)r}{2}$.
By \Cref{lem:GlobalAverage}, the orbit $\mathcal{O}$ containing $f$ is orthomesic.
Thus, \cref{linear_ext} tells us when orthomesy of the total sum statistic is preserved under inflation.
The next proposition provides a similar criterion for orthomesy of the antipodal sum statistic $\mathcal{A}_x$.
In the following, given $x$ an element of a self-dual poset $P$, call  a promotion orbit $\mathcal{O}$ of $\Inc^q(P)$ $x$-stable if, for all $1 \leq k \leq q$, the multiplicity of $k$ in the multiset $\{f(x),f(\kappa(x))
\}_{f \in \mathcal{O}}$ is equal to the multiplicity of $q+1-k$. This is a generalization of the property of being swap-closed; if $\mathcal{O}$ is swap-closed, then it is $x$-stable for all $x \in P$. 

\begin{proposition}\label{prop:Stable}
Let $P$ be a self-dual poset with an order-reversing involution and let $x \in P$. Let $f \in \Inc^q(P)$ be such that $r,\ell,\tau$ from \Cref{thm:OrderComparedToDeflation} satisfy $\gcd(r\ell/q,\tau) = 1$. Let $\mathcal{O}$ be the promotion orbit of elements of $\Inc^q(P)$ containing $f$. If the deflated orbit $\overline{\mathcal{O}}$ is $x$-stable, then $\mathcal{O}$ exhibits orthomesy with respect to the antipodal sum statistic $\mathcal{A}_x$. 
\end{proposition}

\begin{proof}
As in the proof of \Cref{linear_ext}, the condition on the gcd implies that  $\mathcal{O}$ contains every possible pair of a labeling in $\overline{\mathcal{O}}$ and a cyclic rotation of $\Con(f)$. 
We begin by assuming that $\ell$, the period of $\Con(f)$, is equal to $q$.
We can use the description of $\mathcal{O}$ as pairs of content words and packed labelings to break up the computation of $\mathcal{A}_x$ for all labelings in $\mathcal{O}$. 

We assume we have chosen $f$ so that $\Con(f) = (c_1,\ldots,c_q)$ has $c_q = 1$. Define $s_1,\ldots,s_r$ to be the nonnegative integers such that the $r$ values of $1$ in $\Con(f)$ occur at $s_1+1,s_1+s_2 + 2,\ldots,s_1 + \cdots + s_{r-1} + s_r + r$. That is, $\Con(f)$ is $s_1$ 0's followed by a 1, then $s_2$ 0's followed by a 1, and so on.

Let $j$ and $k$ be such that $\overline{f}(x) = j$ and $\overline{f}(\kappa(x)) = k$. With our notation, this means $f(x) = (\sum_{i=1}^j s_i) + j$ and $f(\kappa(x)) = (\sum_{i=1}^k s_i) + k$. For all $t \leq s_1$ (equivalently, for all $t$ such that $c_1 = c_2 = \cdots = c_t = 0$), from \Cref{prop:ProRestricts} we know $\Pro^t(f)(x) = (\sum_{i=1}^j s_i) + j - t$ and $\Pro^t(f)(\kappa(x)) = (\sum_{i=1}^k s_i) + k - t$.  Therefore, we have
\begin{align*}
\sum_{i=0}^{s_1} \mathcal{A}_x(\Pro^i(f)) &= (s_1 + 1)(\sum_{i=1}^{j} s_i + \sum_{i=1}^{k} s_i + j + k) -2 - 4 - \cdots - 2s_1\\
&= (s_1 + 1)(\sum_{i=1}^{j} s_i + \sum_{i=1}^{k} s_i + j + k) -s_1(s_1+1)\\
& = (s_1 + 1)(\sum_{i=1}^{j} s_i + \sum_{i=1}^{k} s_i + j + k-s_1)\\
& = (s_1 + 1)(s_1 + \sum_{i=2}^{j} s_i + \sum_{i=2}^{k} s_i + j + k).\\
\end{align*}

The computation is similar for all other possible deflations. Let $f_2,\ldots,f_\tau$ be the $\tau-1$ other labelings in $\mathcal{O}$ with content word equal to $\Con(f)$. Let $j_z,k_z$ be such that $\overline{f}_z(x) = j_z$ and $\overline{f}_z(\kappa(x)) = k_z$. Let $f_1$ be our original labeling $f$ so that $j_1 = j$ and $k_1 = k$ from the previous summation. Then, we compute
\begin{align}
\sum_{z = 1}^\tau \sum_{i=0}^{s_1} \mathcal{A}_x(\Pro^i(f_z)) = (s_1 + 1)\big(a_1 s_1 + a_2 s_2 + \cdots + a_rs_r + \sum_{z=1}^\tau j_z + k_z)\big)\label{eq:SumFirstSection}
\end{align}
where $a_1 = \tau$ and for $i > 1$, $a_i$ is equal to the number of values in the multiset $\{j_z,k_z\}_{z=1}^\tau$ which are at least $i$. 
Note that the stable condition implies $\sum_{z=1}^\tau (j_z + k_z) = \tau(r+1)$.  This means \begin{align}
 \sum_{i=1}^r a_i=\tau+\sum_{i=2}^r a_i=\tau+\sum_{z=1}^\tau ((j_z-1)+(k_z-1))=\sum_{z=1}^\tau (j_z+k_z)-\tau= \tau (r+1) - \tau = \tau r.  \label{eq:SimplifySumai} 
\end{align}

Let $g$ be the labeling which deflates to $\overline{f}$ and has content word $\rho^{s_1+1}(\Con(f))$. As noted in the first sentence of this proof, we know $g$ is in the orbit $\mathcal{O}$.  By the definition of the numbers $s_i$,  we know that $\mathcal{A}_x(g) = (\sum_{i=2}^{j+1} s_i) + j + (\sum_{i=2}^{k+1} s_i) + k$.

We can compute $\sum_{i=s_1 + 1}^{s_1 + s_2 + 1} \mathcal{A}_x(\Pro^i(g))$ analogously to the computation of $\sum_{i=0}^{s_1} \mathcal{A}_x(\Pro^i(f))$, but with the values $s_1,\ldots,s_r$ cyclically rotated. Therefore, $\sum_{h \in \mathcal{O}} \mathcal{A}_x(h)$ is the result of summing the $r$ terms given by all cyclic shifts of \Cref{eq:SumFirstSection}. Explicitly, this is \begin{align*}
\sum_{h \in \mathcal{O}}\mathcal{A}_x(h) &= (s_1 + 1)(\tau s_1 + a_2 s_2 + \cdots + a_{r-1}s_{r-1} + a_rs_r + \tau (r+1))\\
& + (s_2 + 1)(\tau s_2 + a_2s_3 + \cdots + a_{r-1}s_r + a_r s_1 + \tau (r+1))\\
& \cdots \\
& +(s_r + 1)(\tau s_r + a_2 s_1 + \cdots + a_{r-1}s_{r-2} + a_rs_{r-1} + \tau (r+1)).
\end{align*}

Treating this as a polynomial in the variables $s_1,\ldots,s_r$ and analyzing the coefficients, we have \begin{align*}
\sum_{h \in \mathcal{O}} \mathcal{A}_x(h) &= \tau(s_1^2 + \cdots + s_r^2) + (a_2+a_r)(s_1s_2 + s_2s_3 + \cdots + s_rs_1) \\
&+ (a_3 + a_{r-1})(s_1s_3 + s_2s_4 + \cdots + s_{r-1}s_1) + \cdots
+ \tau(2r+1) (s_1 + \cdots + s_r) + \tau r (r+1),
\end{align*}
where we use \Cref{eq:SimplifySumai} to analyze for the coefficient of linear terms. As $a_r$ is the number of values in $\{j_z,k_z\}_z$ which are equal to $r$, we have $a_r=2\tau - \# \{j_z,k_z\}_z \cap [1,r-1]$ and by the stable condition this is also $2\tau - \# \{j_z,k_z\}_z \cap [2,r]$. This latter quantity is $2\tau - a_2$ because $\# \{j_z,k_z\}_z \cap [2,r]$ is the number of elements in $\{j_z,k_z\}_z$ that are at least 2. This means $a_2 + a_r = 2\tau$. We can proceed similarly for the other coefficients of the squarefree quadratic terms. Therefore, we have \begin{align*}
\sum_{h\in\mathcal{O}}\mathcal{A}_x(h) &= \tau(s_1^2 + \cdots + s_r^2) + 2\tau(s_1s_2 + s_1s_3 + \cdots + s_{r-1}s_r) + (\tau(2r+1) )(s_1 + \cdots + s_r) \\&+ \tau r (r+1)\\
&= \tau (s_1 + s_2 + \cdots + s_r + r)(s_1 + s_2 + \cdots + s_r + r+1)\\
&= \tau q (q+1).
\end{align*}
Since the size of $\mathcal{O}$ is $\tau q$ by Theorem \ref{thm:OrderComparedToDeflation}, the average value of $\mathcal{A}_x$ on $\mathcal{O}$ is $q+1$. By \Cref{lem:GlobalAverage} this is equal to the global average of $\mathcal{A}_x$ on $\Inc^q(P)$. 

Now, if $\ell$ is a proper divisor of $q$, this computation counts each $h\in\mathcal{O}$ exactly $\frac{q}{\ell}$ times. In this case, the size of $\mathcal{O}$ is $\tau \ell$, and the result again follows.
\end{proof}

\begin{example}\label{ex:prop.Stable}
We demonstrate \Cref{prop:Stable} with the following increasing labeling $f$ in $\Inc^8(\calZ_4)$:
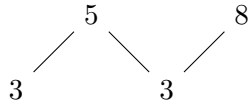
\begin{center}
\begin{tikzpicture}
   \node (A) at (0,0) {3};
    \node (B) at (1,1) {5};
    \node (C) at (2,0) {3};
    \node (D) at (3,1) {8};
    \draw (A) -- (B);
    \draw (B) -- (C);
    \draw (C) -- (D);
\end{tikzpicture}
\end{center}

The deflation of $f$ is:
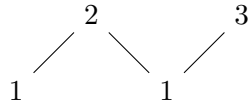
\begin{center}
\begin{tikzpicture}
   \node (A) at (0,0) {1};
    \node (B) at (1,1) {2};
    \node (C) at (2,0) {1};
    \node (D) at (3,1) {3};
    \draw (A) -- (B);
    \draw (B) -- (C);
    \draw (C) -- (D);
\end{tikzpicture}
\end{center}
Its promotion orbit $\mathcal{O}$ in $\Inc^3(\calZ_4)$ is as follows:
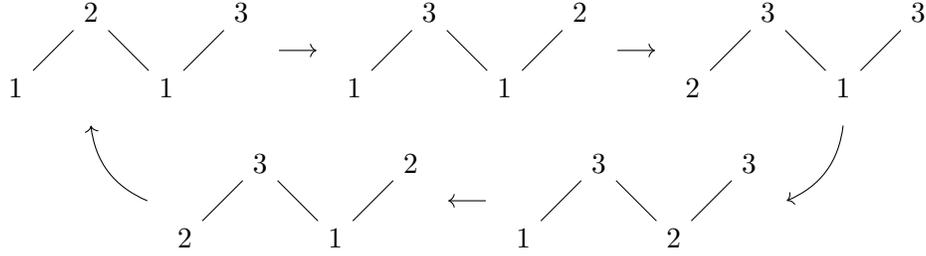
\begin{center}
   \begin{tikzpicture}
   \node (A1) at (0,0) {1};
    \node (B1) at (1,1) {2};
    \node (C1) at (2,0) {1};
    \node (D1) at (3,1) {3};
    \draw (A1) -- (B1);
    \draw (B1) -- (C1);
    \draw (C1) -- (D1);
    \draw[->] (3.5, 0.5) -- (4,0.5);
   \node (A2) at (4.5,0) {1};
    \node (B2) at (5.5,1) {3};
    \node (C2) at (6.5,0) {1};
    \node (D2) at (7.5,1) {2};
    \draw (A2) -- (B2);
    \draw (B2) -- (C2);
    \draw (C2) -- (D2);
    \draw[->] (8, 0.5) -- (8.5,0.5);
   \node (A3) at (9,0) {2};
    \node (B3) at (10,1) {3};
    \node (C3) at (11,0) {1};
    \node (D3) at (12,1) {3};
    \draw (A3) -- (B3);
    \draw (B3) -- (C3);
    \draw (C3) -- (D3);
    \draw (11, -0.5) edge[->,bend left] (10.25,-1.5);
   \node (A4) at (6.75,-2) {1};
    \node (B4) at (7.75,-1) {3};
    \node (C4) at (8.75,-2) {2};
    \node (D4) at (9.75,-1) {3};
    \draw (A4) -- (B4);
    \draw (B4) -- (C4);
    \draw (C4) -- (D4);
   \draw[->] (6.25, -1.5) -- (5.75,-1.5);
   \node (A5) at (2.25,-2) {2};
    \node (B5) at (3.25,-1) {3};
    \node (C5) at (4.25,-2) {1};
    \node (D5) at (5.25,-1) {2};
    \draw (A5) -- (B5);
    \draw (B5) -- (C5);
    \draw (C5) -- (D5);
    \draw (1.75, -1.5) edge[->,bend left] (1,-0.5);
\end{tikzpicture}
\end{center}

 Take $x$ to be the leftmost minimal element, and observe that the multiset $\{\overline{f}(x), \overline{f}(\kappa(x)): f \in \mathcal{O}\} = \{1,1,1, 2,2,2,2, 3,3,3\}$.
Since $1$ and $3$ appear with the same multiplicity, we see that the $\overline{f}$-orbit is $x$-stable.
Moreover, $\gcd(r\ell/q,\tau) = \gcd((3\cdot 8)/8,5)= 1$.
We conclude by \Cref{prop:Stable} that the promotion orbit of $f$ is orthomesic with respect to the antipodal sum statistic $\mathcal{A}_x$.
\end{example}
\section{Orthomesy for $\calZ_4$}
\label{sec:Z4}

In this section, we prove an orthomesy for the antipodal sum statistics of increasing labelings of the poset $\calZ_4$ under promotion. Here there are only two antipodal sum statistics, so we give them their own notation. 
\begin{definition}
    Given an increasing labeling $f\in\Inc^q(\calZ_n)$, we define the exterior antipodal sum $\ase(f)$ as the sum of the labels on the leftmost and rightmost elements. Given  $f\in\Inc^q(\calZ_4)$, we define the interior antipodal sum $\asi(f)$ as the sum of the labels on the other two elements.
\end{definition}

\begin{example}\label{ex:as_e}
Consider the following increasing labelings of $\calZ_4$ and $\calZ_6$, which we refer to as $f$ and $g$ respectively.
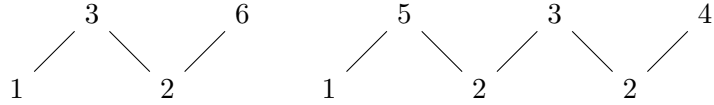
\begin{center}
\begin{tabular}{ccc}
\begin{tikzpicture}
    \node (A) at (0,0) {1};
    \node (B) at (1,1) {3};
    \node (C) at (2,0) {2};
    \node (D) at (3,1) {6};
    \draw (A) -- (B);
    \draw (B) -- (C);
    \draw (C) -- (D);
\end{tikzpicture}&&
\begin{tikzpicture}
    \node (A) at (0,0) {1};
    \node (B) at (1,1) {5};
    \node (C) at (2,0) {2};
    \node (D) at (3,1) {3};
    \node (E) at (4,0) {2};
    \node (F) at (5,1) {4};
    \draw (A) -- (B);
    \draw (B) -- (C);
    \draw (C) -- (D);
    \draw (D) -- (E);
    \draw (E) -- (F);
\end{tikzpicture}
\end{tabular}
\end{center}
The exterior antipodal sum is $1+6=7$ for $f$ and $1+4=5$ for $g$. The interior antipodal sum is $3+2=5$ for $f$.
\end{example}

We next define a pattern avoidance condition on labelings of fence posets.  This is analogous to the idea of pattern avoidance in permutations.

\begin{definition}\label{def:pattern}
Consider an increasing labeling of a fence poset $P$ with $n$ elements. Construct the word $u=u_1u_2\dots u_n$ by reading the labels from left to right.  We say the labeling \emph{contains a permutation pattern} $\pi=\pi_1\pi_2\dots\pi_k$ if there exist $i_1<i_2<\dots<i_k$ such that $u_{i_1},u_{i_2},\dots,u_{i_k}$ is in the same relative order as $\pi$.  
    If a labeling does not contain a pattern $\pi$, we say it \emph{avoids} $\pi$. 
    Further, we say a subset of $\Inc^q(P)$ avoids $\pi$ if every labeling in the subset avoids $\pi$.
\end{definition}

\begin{example}
\label{ex:labelings} 
The words for the labelings $f$ and $g$ from \Cref{ex:as_e} are $1326$ and $152324$, respectively. We see, for example, that $f$ avoids pattern 312 but $g$ does not as it has the subword $524$.
\end{example}

\begin{definition}\label{def:Balanced}
    We say an element of $\Inc^q(\calZ_4)$ with labels $w< x < y < z$ is \emph{imbalanced} if  $x-w\neq z-y$ and $w-z+q\neq y-x$.  We say it is \emph{balanced} otherwise. 
\end{definition}

\begin{example}
Consider the following increasing labelings in $\Inc^6(\calZ_4)$.

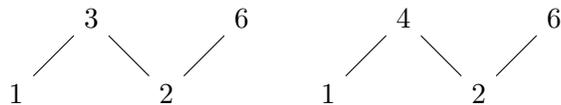
\begin{center}
\begin{tabular}{ccc}
\begin{tikzpicture}
    \node (A) at (0,0) {1};
    \node (B) at (1,1) {3};
    \node (C) at (2,0) {2};
    \node (D) at (3,1) {6};
    \draw (A) -- (B);
    \draw (B) -- (C);
    \draw (C) -- (D);
\end{tikzpicture}&&
\begin{tikzpicture}
   \node (A) at (0,0) {1};
    \node (B) at (1,1) {4};
    \node (C) at (2,0) {2};
    \node (D) at (3,1) {6};
    \draw (A) -- (B);
    \draw (B) -- (C);
    \draw (C) -- (D);
\end{tikzpicture}
\end{tabular}
\end{center}

Using the notation in \Cref{def:Balanced}, for the first labeling we set $w = 1, x = 2, y = 3, z = 6$. We check $w-z+q = 1$ and $y - x = 1$ so that the first labeling is balanced.

For the second labeling, we set $w = 1, x = 2, y = 4, z = 6$. Since $2-1 \neq 6-4$ and $1-6+6 \neq 4-2$, this labeling is imbalanced.
\end{example}

\begin{proposition}\label{prop:balanced}
    Given a promotion orbit $\mathcal{O}$ of $\Inc^q(\calZ_4)$ where some element of $\mathcal{O}$ (equivalently, each element of $\mathcal{O}$) has distinct labels, either all labelings in the orbit are imbalanced or all labelings in the orbit are balanced.
\end{proposition}

\begin{proof}
Let $f\in\Inc^q(\calZ_4)$ with labels $w< x < y < z$ and suppose $f$ is imbalanced. 
Recall from \Cref{lem:ProRotateContent} that promotion shifts the binary content word.  Thus, if $w\neq1$ then the spacing between non-zero entries in the binary content word of $\Pro(f)$ is the same as in the binary content word of $f$.
If $w=1$, then $\Pro(f)$ has labels $x-1 < y-1 < z-1 < q$.
Since $x-w\neq z-y$ and $w-z+q\neq y-x$ and $w=1$, we know $y-x\neq q-z+1$ and $x-1\neq z-y$.
This means that in either case, $\Pro(f)$ is imbalanced.
Since applying promotion to an imbalanced labeling results in another imbalanced labeling, if an orbit contains any imbalanced labelings, it must contain only imbalanced labelings.
\end{proof}

\begin{remark}
\Cref{prop:balanced} can be false if the orbit $\mathcal{O}$ contains an element whose labels are not distinct. For example, consider the orbit in $\Inc^4(\calZ_4)$ containing the labeling $f$ with word $2414$. Then $f$ is imbalanced because $x-w = 2-1=1$ while $z-y = 4-4=0$, and $w-z+q=1-4+4=1$ while $y-x=4-2=2$.
However $\Pro(f)$ is balanced.
While $x-1, y-1, z-1$, and $q$ are labels in $\Pro(f)$, the content word only keeps track of the labels, not the multiplicities.  While $y=z$ appears as a label with multiplicity 2 in $f$, $y-1=z-1$ only appears as a label with multiplicity 1 in $\Pro(f)$.  Similarly, $w$ appears as a label with multiplicity 1 in $f$ where $q$ appears as a label with multiplicity 2 in $\Pro(f)$.
That is, the labels of $\Pro(f)$ are $w'=x-1$, $x'=y-1=z-1$, $y'=q$, and $z'=q$.
As a result, the equation $x-1 \ne z-y$ in the proof becomes $x-1\ne q-(y-1)$, which is false.
\end{remark}

Following the above proposition, we call an promotion orbit of $\Inc^q(\calZ_4)$ where some element has distinct labels \emph{balanced} when all its elements are balanced.

The following theorem is an orthomesy result that is the main aim of the rest of the section.
\begin{theorem}\label{thm:near-homomesy}
    Let $P=\calZ_4$ and $x\in P$. Let $\mathcal{O}$ be a promotion orbit of $\Inc^q(P)$. Then $\mathcal{O}$ exhibits orthomesy with respect to the antipodal sum statistic $\as_x$ if and only if $\mathcal{O}$ avoids $1324$ or is balanced. 
\end{theorem}

\begin{proof}
    In \Cref{thm:near-homomesy-ABD} we will show that if $\mathcal{O}$ avoids 1324 then $\mathcal{O}$ exhibits orthomesy with respect to the antipodal sum statistic $\as_x$.  In \Cref{prop:near-homomesy-C} we will prove that if $\mathcal{O}$ contains 1324 then the average value of $\as_x$ on $\mathcal{O}$ is $q+1$ if and only if $\mathcal{O}$ is balanced.  Part (1) of \Cref{lem:GlobalAverage} says that the global average of $\as_x$ is equal to $q+1$, implying the result.
\end{proof}

\begin{proposition}\label{thm:near-homomesy-ABD}
    Let $P=\calZ_4$ and $x\in P$. Let $\mathcal{O}$ be a promotion orbit of $\Inc^q(P)$ that avoids $1324$. Then $\mathcal{O}$ exhibits orthomesy with respect to the antipodal sum statistic $\as_x$.
\end{proposition}

\begin{proof}
It will be convenient for us to be able to understand the packed labelings for $P$ explicitly. These sit in four orbits.  The first orbit, which we will call orbit $A$, consists of one labeling:
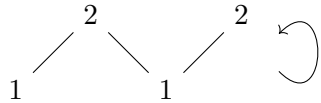
\begin{center}
\begin{tikzpicture}
    \node (A) at (0,0) {1};
    \node (B) at (1,1) {2};
    \node (C) at (2,0) {1};
    \node (D) at (3,1) {2};
    \draw (A) -- (B);
    \draw (B) -- (C);
    \draw (C) -- (D);
    \draw (3.5,0.25) edge[->, out = -45, in = 45, looseness=5] (3.5,0.75);
\end{tikzpicture}
\end{center}
The second orbit, which we will call orbit $B$, is the orbit shown in \Cref{ex:prop.Stable}.  The third orbit, which we will call orbit $C$, consists of the following two labelings:
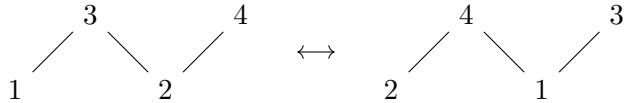
\begin{center}
\begin{tikzpicture}
    \node (A1) at (0,0) {1};
    \node (B1) at (1,1) {3};
    \node (C1) at (2,0) {2};
    \node (D1) at (3,1) {4};
    \draw (A1) -- (B1);
    \draw (B1) -- (C1);
    \draw (C1) -- (D1);
   \node (A2) at (5,0) {2};
    \node (B2) at (6,1) {4};
    \node (C2) at (7,0) {1};
    \node (D2) at (8,1) {3};
    \draw (A2) -- (B2);
    \draw (B2) -- (C2);
    \draw (C2) -- (D2);
   \draw[->]  (4,0.5) -- (4.25,0.5);
   \draw[->]  (4,0.5) -- (3.75,0.5);
\end{tikzpicture}
\end{center}
Finally, the last orbit, which we will call orbit $D$, is the orbit shown in \Cref{ex:linear_ext}. 

From \Cref{prop:ProRestricts}, the elements of $\mathcal{O}$ deflate to elements of exactly one of the packed orbits, which we call $\overline{\mathcal{O}}$.  Then from \Cref{cor:can-find-lift}, the elements of $\mathcal{O}$ deflate to \emph{all} of the elements of $\overline{\mathcal{O}}$.  Since $\mathcal{O}$ avoids $1324$, we know $\overline{\mathcal{O}}\neq C$.

Suppose first that $\overline{\mathcal{O}}=A$. Given $f \in \mathcal{O}$, we can decompose $f$ into its deflation $1212$ and $\Con(f) = (c_1,\ldots,c_q)$. From \Cref{lem:SwapOnContent} and \Cref{lem:SwapAndDeflation}, we know that $\swap(f)$ has the reverse content word, $\rev(\Con(f))$, and same deflation. Since $\Con(f)$ is a binary word containing only two 1's, $\rev(\Con(f))=\rho^i(\Con(f))$ for some $i$.  By \Cref{lem:ProRotateContent}, $\Pro^i(f)$ has content word $\rho^i(\Con(f))=\rev(\Con(f))$ and by \Cref{prop:ProRestricts}, $\Pro^i(f)$ has deflation $1212$.  \Cref{remark:defl_con} tells us this means $\Pro^i(f)=\swap(f)$.  Thus, $\swap(f) \in \mathcal{O}$, and by \Cref{cor:TwoTypesOfOrbits}, $\swap(\mathcal{O}) = \mathcal{O}$. The claim holds by \Cref{thm:swap_orthomesy}.

Next, suppose that $\overline{\mathcal{O}}=B$. Let $f \in \mathcal{O}$ and let $\ell$ be the period of $\Con(f)$. Since $f$ uses three distinct numbers, $\ell$ must either be $q$ or $\frac{q}{3}$.  In either case, $\gcd(\frac{3\ell}{q},5) = 1$. We also observe that $\overline{\mathcal{O}} = B$ is $x$-stable (Example 6.5 shows this when $x$ is an antipodal element). Therefore, the result follows in this case by \Cref{prop:Stable}.

Finally, suppose $\overline{\mathcal{O}}=D$. We will use the same method in this case as we did for when $\overline{\mathcal{O}}=B$.  Let $f \in \mathcal{O}$ and let $\ell$ be the period of $\Con(f)$. Since $f$ uses four distinct numbers, $\ell$ must be $q$, $\frac{q}{2}$ or $\frac{q}{4}$.  In all cases, $\gcd(\frac{4\ell}{q},3) = 1$. Since $\overline{\mathcal{O}} = D$ is $x$-stable, the result follows by \Cref{prop:Stable}.
\end{proof}

\begin{proposition}\label{prop:near-homomesy-C}
    Let $P=\calZ_4$ and $x\in P$. If $\mathcal{O}$ is a promotion  orbits of $\Inc^q(P)$ that contains $1324$ then the average value of the antipodal sum statistic $\as_x$ equals $q+1$ if and only if $\mathcal{O}$ is balanced.
\end{proposition}

\begin{proof}
By \Cref{prop:ProRestricts}, there exists $f\in\mathcal{O}$ such that $\bar{f} = 1324$ and if $\mathbf{c} = (c_1,\ldots,c_q) = Con(f)$, then $c_q = 1$. We set $\alpha, \beta,\gamma$ to be the nonnegative integers such that $c_{\alpha + 1}=c_{\alpha + \beta + 2} = c_{\alpha+\beta+\gamma+3}=c_q=1$ and $\alpha+\beta+\gamma+3<q$. Let $\delta = q - 4 - \alpha - \beta-\gamma$. 
With this notation, $\mathcal{O}$ being balanced is equivalent to having $\beta = \delta$ or $\alpha = \gamma$.

First assume $\as_x$ is the exterior antipodal sum statistic $\ase$. Our proof method will be similar to that of \Cref{prop:Stable}. We will begin by adding up the  exterior antipodal sums of $\{(1324, \rho^i(\mathbf{c})): 0 \leq i \leq \alpha\}$ where we are identifying each ordered pair $(g,v)$ with the increasing labeling $f$ such that $\overline{f}=g$ and $\Con(f)=v$. Note that for these values of $i$, $\Pro^i(f)= \Pro^i(\overline{f}, \mathbf{c}) = (\overline{f}, \rho^i(\mathbf{c}))$ from \Cref{lem:ProRotateContent} and \Cref{prop:ProRestricts}. To compute $\ase(1324, \rho^i(\mathbf{c}))$ where $0 \leq i \leq \alpha$, we add the index of the first and last nonzero entries of $\rho^i(\mathbf{c})$.  That is, $\ase(1324, \rho^i(\mathbf{c}))=(\alpha+1-i)+(q-i)$. Then, the sum of all these exterior antipodal sums is given by \begin{align*}[(\alpha + 1)+(q)] + [(\alpha)+ (q-1)] + \cdots + [(1)+(q-\alpha)] &= \sum_{j=0}^\alpha (q-\alpha+1) + 2j\\ &= (q-\alpha+1)(\alpha+1) + (\alpha+1)\alpha\\ &=(q+1)(\alpha+1)\\
&= (\alpha + \beta + \gamma + \delta + 5)(\alpha + 1).\end{align*}

Since $\mathbf{c}$ is a binary word with four 1's,  it could have no symmetry, 2-fold symmetry, or 4-fold symmetry.  If $\mathbf{c}$ has 2-fold or 4-fold symmetry then $\alpha=\gamma$, $\beta=\delta$ and, applying \Cref{lem:ProRotateContent} and \Cref{prop:ProRestricts}, we see $\mathcal{O} = \{(1324,\rho^i(\mathbf{c})): 0 \leq i \leq \alpha\} \cup \{(2413, \rho^i(\mathbf{c})):\alpha+1 \leq i \leq \alpha + \beta + 1\}$.
The sum of the exterior antipodal sums for the labelings in the first set is computed above; for the second set, it is $(\beta + 2\gamma + \delta + 5)(\beta+1)$ by a similar computation. 
Therefore, in this case the sum of $\as_e$ over all elements of $\mathcal{O}$ is $(2\alpha + 2\beta + 5)(\alpha + \beta + 2)$, which is equal to $(q+1)(\frac{q}{2}) = (q+1) \vert \mathcal{O} \vert$ since $\alpha + \beta + 2 = \frac{q}{2}=|\mathcal{O}|$.

In the case when $\mathbf{c}$ has no symmetry, by \Cref{lem:ProRotateContent} and \Cref{prop:ProRestricts}, \begin{align*}
\mathcal{O} &= \{(1324,\rho^i(\mathbf{c}): 0 \leq i \leq \alpha\}\ \cup\ \{2413, \rho^i(\mathbf{c}):\alpha+1 \leq i \leq \alpha + \beta + 1\}\ \cup \\
&\ \quad\{(1324,\rho^i(\mathbf{c}): \alpha + \beta + 2 \leq i \leq \alpha+\beta + \gamma + 2\}\ \cup\ \{2413, \rho^i(\mathbf{c}):\alpha+\beta+\gamma+3 \leq i \leq q-1\}.
\end{align*}
Adding the exterior antipodal sums of all these labelings gives 
$$(\alpha + \beta + \gamma + \delta + 5)(\alpha+1) + ( \beta + 2\gamma + \delta + 5)(\beta+1)+ (\alpha + \beta + \gamma + \delta + 5)(\gamma+1)+ ( \delta + 2\alpha + \beta + 5)(\delta+1)$$
which simplifies to 
\begin{align} (\alpha + \beta + \gamma + \delta + 4)(\alpha + \beta + \gamma + \delta + 5) + (\alpha-\gamma)(\delta-\beta)=q(q+1) + (\alpha-\gamma)(\delta-\beta)\label{eq:NotAvoiding1324}.
\end{align}
Since $|\mathcal{O}|=q$ in this case, the average value of $\as_e$ on the elements of $\mathcal{O}$ is $q+1$ exactly when $\alpha = \gamma$ or $\beta=\delta$.

Considering both cases, we see that the average value of $\as_e$ on the elements of $\mathcal{O}$ is $q+1$ exactly when $\alpha=\gamma$ or when $\beta=\delta$.  In other words, the average value of $\as_e$ on the elements of $\mathcal{O}$ is $q+1$ exactly when $\mathcal{O}$ is balanced.

Now let $\as_x$ be the interior sum statistic $\asi$.
Since the interior antipodal sum of $\{(1324, \rho^i(\mathbf{c})): 0 \leq i \leq \alpha\}$ is the same as the exterior antipodal sum of $\{(2413, \rho^i(\mathbf{c})): 0 \leq i \leq \alpha\}$ and the interior antipodal sum of  $\{(2413, \rho^i(\mathbf{c})): \alpha+1 \leq i \leq \alpha+\beta+1\}$ is the same as the exterior antipodal sum of $\{(1324, \rho^i(\mathbf{c})): \alpha+1 \leq i \leq \alpha+\beta+1\}$, it follows from the proof for $\ase$ that the interior antipodal sum has average value $q+1$ on the orbit $\mathcal{O}$ exactly when $\mathcal{O}$ is balanced.
\end{proof}

\begin{remark}
    From \Cref{cor:can-find-lift} and the fact that $1324$ and $2413$ form an orbit in $\Inc^4(\calZ_4)$, we can see that \Cref{thm:near-homomesy} could have been stated equivalently as either of the following:
    \begin{itemize}
        \item Let $P=\calZ_4$ and $x\in P$. Let $\mathcal{O}$ be a promotion orbit of $\Inc^q(P)$. Then $\mathcal{O}$ exhibits orthomesy with respect to the antipodal sum statistic $\as_x$ if and only if $\mathcal{O}$ avoids $2413$ or is balanced.
        \item Let $P=\calZ_4$ and $x\in P$. Let $\mathcal{O}$ be a promotion orbit of $\Inc^q(P)$. Then $\mathcal{O}$ exhibits orthomesy with respect to the antipodal sum statistic $\as_x$ if and only if $\mathcal{O}$ avoids $1324$ and $2413$ or is balanced.
    \end{itemize}
\end{remark}

Note that for $q = 2$ or $q = 3$, there are no orbits that contain a 1324 pattern and for $q=4,5$ all such orbits are balanced. Thus in these cases, both antipodal sum statistics are homomesic on promotion orbits of $\Inc^q(\calZ_4)$.  For $q\geq6$, neither antipodal sum statistic is homomesic on promotion orbits of $\Inc^q(\calZ_4)$.  For example, for $q=6$ the 190 labelings of $\Inc^q(\calZ_4)$ are partitioned into 16 promotion orbits.  Of these orbits, 14 meet the conditions in \Cref{thm:near-homomesy} and therefore have average exterior antipodal sum and interior antipodal sum of $q+1=7$.  There are two orbits that do not meet the conditions in \Cref{thm:near-homomesy}.  One of these orbits  has average exterior antipodal sum of $41/6$ and average interior antipodal sum $43/6$; the other has these values reversed.  Note that $\swap$ sends the elements of each of these ``bad" orbits to each other.

Our final result gives a homomesy on $\Inc^q(Z_4)$, as a corollary of the orthomesy of \Cref{thm:near-homomesy}. 
\begin{corollary}
\label{Z4_totalsum}
    Promotion on $\Inc^q(\calZ_4)$ is $2(q+1)$-mesic with respect to the total sum statistic.
\end{corollary}

\begin{proof}
If $\mathcal{O}$ is an orbit which avoids $1324$ or is balanced, then by \Cref{thm:near-homomesy}, the average value on both antipodal sum statistics on $\mathcal{O}$ is $q+1$. Therefore, since given $f \in \mathcal{O}$, $\tot(f) = \mathcal{A}_e(f) + \mathcal{A}_i(f)$, the result follows immediately.

It remains to consider an orbit $\mathcal{O}$ which does not avoid 1324 and is not balanced. Therefore, there exists $f \in \mathcal{O}$ such that $\overline{f} = 1324$ and $\mathbf{c} = (c_1,\ldots,c_q) = \Con(f)$ has $c_q = 1$. As previously, define  $\alpha,\beta,\gamma$ to be the nonnegative integers such that $c_{\alpha+1} = c_{\alpha+\beta+2} = c_{\alpha+\beta+\gamma+3}= 1$ and $\alpha+\beta+\gamma+3 < q$. Let $\delta = q - 4 - \alpha - \beta - \gamma$. Since $f$ is not balanced,  we have $\alpha \neq \gamma$ and $\beta \neq \delta$.

Recall that in \Cref{eq:NotAvoiding1324} from \Cref{prop:near-homomesy-C}, we gave the simplified expression from adding up all antipodal sums $\mathcal{A}_x$ of $\mathcal{O}$.  Combining the two distinct antipodal pairs, we have 
\begin{align*}
\sum_{f\in O} \tot(f) &= (\alpha + \beta + \gamma + \delta + 4)(\alpha + \beta + \gamma + \delta + 5) + (\alpha-\gamma)(\delta-\beta)\\ 
&+ (\alpha + \beta + \gamma + \delta + 4)(\alpha + \beta + \gamma + \delta + 5) + (\beta-\delta)(\alpha-\gamma)\\
&= 2(\alpha + \beta + \gamma + \delta + 4)(\alpha + \beta + \gamma + \delta + 5)\\
&=2q(q+1)
\end{align*}
as desired.
\end{proof}

\section{Future work}\label{sec:future}

This work raises several natural questions and suggests promising directions for further investigation.

Overall, we would like to find other, natural instances of orthomesy. A natural starting  point would be to look for remarks in the literature, such as \cite[Propositions 38 and 39]{Vorland} of non-homomesic pairs of maps and statistics and try to characterize which subset of orbits exhibits orthomesy. One could also undertake a systematic search
 using FindStat \cite{FindStat}, as was done \cite{FindStat1,FindStat2} for other phenomena in dynamical algebraic combinatorics. 

It would also be interesting to extend our results on small zig-zag posets. For instance, is there a pattern to the order of promotion on $\calZ_n$ as $n$ grows? As we mentioned in \Cref{rmk:OrderOnZnLargern}, the order of promotion on $\Inc^q(\calZ_6)$ for $q \geq 6$ appears to be $13090q$. Since this number factors into small primes, namely  $13090 = 2 \times 5 \times 7 \times 11 \times 17$, there is hope that in general there will be a nice formula for the order of promotion on $\calZ_n$. 

\Cref{thm:near-homomesy} gives conditions under which a promotion orbit of $\calZ_4$ exhibits orthomesy with respect to the antipodal sum statistic. It is a natural question to ask whether a similar result can be found for $\calZ_n$ with $n\neq 4$.
We restrict our attention to $n$ even, since antipodal sum is not defined if $n$ is odd. 
The next natural poset to consider is $\calZ_6$.
We can use both \Cref{linear_ext} and \Cref{prop:Stable} to find packed increasing labelings of $\mathcal{Z}_6$ whose inflations are orthomesic (with respect to the total sum statistic or the antipodal sum statistic).
For example the two following packed increasing labelings in $\Inc^5(\mathcal{Z}_6)$ have promotion orbits with cardinality equal to 17 and which are $x$-stable, when we take $x$ to be the leftmost minimal element.
Thus, \cref{prop:Stable} tells us that any inflation of either orbit which makes ${\gcd(5\cdot \ell/q, 17)=1}$ is also orthomesic, where $\ell$ is the period of the content vector of the inflation. (In this case, the orbit of any inflation is orthomesic, since $q/\ell$ divides $5$, so it is either $5$ or $1$. That is, $\ell/q$ equals either $1$ or $1/5$, making the gcd $1$.)
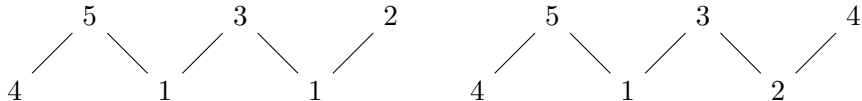
\begin{center}
\begin{tabular}{ccc}
\begin{tikzpicture}
    \node (A) at (0,0) {4};
    \node (B) at (1,1) {5};
    \node (C) at (2,0) {1};
    \node (D) at (3,1) {3};
    \node (E) at (4,0) {1};
    \node (F) at (5,1) {2};
    \draw (A) -- (B);
    \draw (B) -- (C);
    \draw (C) -- (D);
    \draw (D) -- (E);
    \draw (E) -- (F);
\end{tikzpicture}&&
\begin{tikzpicture}
    \node (A) at (0,0) {4};
    \node (B) at (1,1) {5};
    \node (C) at (2,0) {1};
    \node (D) at (3,1) {3};
    \node (E) at (4,0) {2};
    \node (F) at (5,1) {4};
    \draw (A) -- (B);
    \draw (B) -- (C);
    \draw (C) -- (D);
    \draw (D) -- (E);
    \draw (E) -- (F);
\end{tikzpicture}
\end{tabular}
\end{center}
More work is needed for a full classification of orthomesic promotion orbits with respect to the antipodal sum statistic.

One could similarly ask to classify orbits exhibiting total sum orthomesy for $\calZ_n$. The total sum homomesy fails for $\Inc^3(\calZ_3)$, which one can readily check by hand. While \Cref{Z4_totalsum} shows that all orbits of $\Inc^q(\calZ_4)$ are orthomesic with respect to the total sum statistic, this is not true for $\calZ_6$. The smallest $q$ where there exist non-orthomesic orbits in $\Inc^q(\calZ_6)$ is $q = 5$. 
In this case, there are $707$ increasing labelings and $31$ orbits, $27$ of which have total sum average equal to the global average of $18$. The remaining orbits are of sizes $[35, 35, 22, 22]$ with total sum averages $[628/35, 632/35, 391/22, 401/22]$. A future direction could be to classify orbits in $\Inc^q(\calZ_n)$ which exhibit total sum orthomesy. 

Finally, in order to characterize promotion orbits of $\Inc^q(\calZ_4)$ which exhibit antipodal sum orthomesy, we introduced the notion of pattern avoidance for increasing labelings of zig-zagag posets. Pattern avoidance of posets is a rich topic of study; one can for instance see the extensive database in \cite{PatternDatabase}. 
Pattern avoiding linear extensions of zig-zag posets are alternating permutations that also avoid a given pattern; these have been a fruitful area of research (see e.g.~\cite{Lewis2011,Yan2013}). It is possible that pattern avoiding increasing labelings of zig-zagag posets would also be an interesting new direction of study.

\section*{Acknowledgments}
This project began at the January 2024 Community in Algebraic and Enumerative Combinatorics workshop at the Banff International Research Station. We would like to thank BIRS and the organizers for the conducive research environment and the opportunity to participate.

The authors thank the anonymous referees for helpful comments, including pointing us to \cite[Fig.~3.69]{EC1}. We also thank Corey Vorland for writing  SageMath~\cite{sagemath} code  that was helpful in this research and Ben Adenbaum, Sergi Elizalde, and Oliver Pechenik for helpful discussions. 

E.~Banaian was supported by Research Project 2 from the Independent Research Fund Denmark (grant no.~1026-00050B).
J.~Striker was supported by Simons Foundation gift MP-TSM-00002802 and NSF grant DMS-2247089.

\bibliographystyle{amsalpha}
\bibliography{bib.bib}

\end{document}